\documentclass[a4paper,11pt]{amsart}
\usepackage{amsfonts}
\usepackage{amssymb}
\usepackage[utf8]{inputenc}
\usepackage{amsmath}
\usepackage{pdflscape}
\usepackage{graphicx}
\usepackage[colorlinks=true, linkcolor=red, citecolor=blue]{hyperref}
\usepackage[]{epsfig}
\usepackage[]{pstricks}
\usepackage{tikz}

\newcommand{\R}{{\mathbb R}}
\newcommand{\T}{{\mathbb T}}

\newcommand{\ft}{{\mathcal{F}}}

\newcommand{\wt}{\widetilde}
\newcommand{\wh}{\widehat}

\newtheorem{theorem}{Theorem}[section]

\newtheorem{lemma}[theorem]{Lemma}

\newtheorem{corollary}[theorem]{Corollary}

\newtheorem{proposition}[theorem]{Proposition}

\theoremstyle{definition}

\theoremstyle{remark}
\newtheorem*{remarks}{Remarks}

\makeatletter
\@namedef{subjclassname@2020}{\textup{2020} Mathematics Subject Classification}
\makeatother

\newcommand{\remove}[1]{ }

\def\R{\mathbb R}
\def\be{\begin{equation}}
\def\ee{\end{equation}}
\def\ba{\begin{eqnarray}}
\def\ea{\end{eqnarray}}

\numberwithin{equation}{section}
\begin{document}
\title[Controllability with an integral condition]{Control of Kawahara equation with overdetermination condition: The unbounded cases}
\author[Capistrano--Filho]{Roberto de A. Capistrano--Filho*}
\thanks{*Corresponding author: roberto.capistranofilho@ufpe.br}
\address{Departamento de Matemática, Universidade Federal de Pernambuco (UFPE), 50740-545, Recife-PE, Brazil.}
\email{roberto.capistranofilho@ufpe.br}
\author[de Sousa]{Luan S. de Sousa}
\address{Departamento de Matemática, Universidade Federal de Pernambuco (UFPE), 50740-545, Recife-PE, Brazil.}
\email{luan.soares@ufpe.br}
\author[Gallego]{Fernando A. Gallego}
\address{Departamento de Matem\'a1tica, Universidad Nacional de Colombia (UNAL), Cra 27 No. 64-60, 170003, Manizales, Colombia}
\email{fagallegor@unal.edu.co}

\begin{abstract}
In this manuscript we consider the internal control problem for the fifth order KdV type equation, commonly called the Kawahara equation, on unbounded domains. Precisely, under certain hypotheses over the initial and boundary data, we are able to prove that there exists an internal control input such that solutions of the Kawahara equation satisfies an \textit{integral overdetermination condition}. This condition is satisfied when the domain of the Kawahara equation is posed in the real line, left half-line and right half-line. Moreover, we are also able to prove that there exists a minimal time in which the  integral overdetermination condition is satisfied.  Finally,  we show a type of exact controllability associated with the ``mass" of the Kawahara equation posed in the half-line. 
\end{abstract}
\maketitle

\section{Introduction}\label{Sec0} 
\subsection{Model under consideration} Water wave systems are too complex to easily derive and rigorously from it relevant qualitative information on the dynamics of the waves. Alternatively, under suitable assumption on amplitude, wavelength, wave steepness and so on, the study on asymptotic models for water waves has been extensively investigated to understand the full water wave system, see, for instance, \cite{ASL,BCS, BCL, BLS, Lannes, Saut} and references therein for a rigorous justification of various asymptotic models for surface and internal waves.

Formulating the waves as a free boundary problem of the incompressible, irrotational Euler equation in an appropriate non-dimensional form, one has two non-dimensional parameters $\delta := \frac{h}{\lambda}$ and $\varepsilon := \frac{a}{h}$, where the water depth, the wavelength and the amplitude of the free surface are parameterized as $h, \lambda$ and $a$, respectively. Moreover, another non-dimensional parameter $\mu$ is called the Bond number, which measures the importance of gravitational forces compared to surface tension forces. The physical condition $\delta \ll 1$ characterizes the waves, which are called long waves or shallow water waves, but there are several long wave approximations according to relations between $\varepsilon$ and $\delta$. 

In this spirit, when we consider  $\varepsilon = \delta^2 \ll 1$ and $\mu \neq \frac13$, we are dealing with the Korteweg-de Vries (KdV) equation. Under this regime, Korteweg and de Vries \cite{Korteweg}\footnote{This equation was firstly introduced by Boussinesq \cite{Boussinesq}, and Korteweg and de Vries rediscovered it twenty years later. Details can be found in \cite{CaZh1} and the reference therein.} derived the following equation well-known as a central equation among other dispersive or shallow water wave models called the KdV equation
\[\pm2 u_t + 3uu_x +\left( \frac13 - \mu\right)u_{xxx} = 0.\] 

Another alternative is to treat a new formulation, that is, when $\varepsilon = \delta^4 \ll 1$ and $\mu = \frac13 + \nu\varepsilon^{\frac12}$, and in connection with the  critical Bond number $\mu = \frac13$, to generate the so-called equation Kawahara equation. That equation was derived by Hasimoto and Kawahara \cite{Hasimoto1970,Kawahara} as a fifth-order KdV equation and take the form \[\pm2 u_t + 3uu_x - \nu u_{xxx} +\frac{1}{45}u_{xxxxx} = 0.\]

Our main focus is to investigate a type of controllability for the higher-order KdV type equation. We will continue working with an \textit{integral overdetermination condition} started in \cite{CaSo} however in another framework, to be precise, on an \textit{unbounded domain}. To do that,  consider the initial boundary value problem (IBVP)
\begin{equation}\label{IBVP}
\left\lbrace
\begin{array}{llr}
u_{t} + \alpha u_{x} + \beta u_{xxx} + \xi u_{xxxxx} + uu_{x}= f_{0}(t)g(t,x) & \mbox{in} \ [0,T]\times\mathbb{R}^{+}, \\ 
u(t,0)=h_1(t), \ u_{x}(t,0)= h_2(t) & \mbox{on} \ [0,T], \\
 u(0,x) = u_{0}(x) & \mbox{in} \ \mathbb{R}^{+},
\end{array}\right. 
\end{equation} 
where $\alpha$,  $\beta$ and $\xi$ are real number, $u= u(t,x)$, $g=g(x,t)$ and $h_i=h_i(t)$, for $i=1,2$, are well-known function and $f_{0}=f_0(t)$ is a control input.  It is important mention that \eqref{IBVP} is called KdV and Kawahara equation when $\xi=0$ and $\xi=-1$, respectively. 

\subsection{Framework of the problems} In this work we will be interested with a kind of a internal control property to the Kawahara equation when an \textit{integral overdetermination condition},  on unbounded domain, is required, namely
\begin{equation}\label{cond. de con.2}
\int_{\mathbb{R}^{+}}^{}u(t,x)\omega(x)dx= \varphi(t), \ t \in [0,T],
\end{equation}
where $\omega$ and $\varphi$ are some known functions. To present the problems under consideration, take the following unbounded domain
 $Q^+_T = (0, T) \times \mathbb{R}^+$, where $T$ is a positive number, consider the  boundary functions $\mu$ and $\nu$, and a source term $f=f(t,x)$ with a special form, to be specify latter. Thus, let us deal with the following system 
\begin{equation}\label{2}
\left\lbrace
\begin{array}{llr}
u_{t} + \alpha u_{x} + \beta u_{xxx} - u_{xxxxx} + uu_{x}= f(t,x)& \mbox{in} \ [0,T]\times\mathbb{R}^{+}, \\ 
u(t,0)=\mu(t), \ u_{x}(t,0)= \nu(t) & \mbox{on} \ [0,T], \\
 u(0,x) = u_{0}(x) & \mbox{in} \ \mathbb{R}^{+},
\end{array}\right. 
\end{equation} 

Therefore, the goal of the article is concentrated on proving an \textit{overdetermination control problem}. Precisely, we want to prove that if $f$ take the following special form \begin{equation}\label{f} f(t,x)=f_0(t)g(t,x), \quad (t,x)\in Q^+_T, \end{equation} the solution of \eqref{2} satisfies the integral overdetermination condition \eqref{cond. de con.2}. In other words, we have the following issue.

\vspace{0.2cm}
\noindent\textbf{Problem $\mathcal{A}$:} For given functions $u_0$, $\mu$, $\nu$ and $g$ in some appropriated spaces, can we find an internal control $f_0$ such that the solution associated to the equation \eqref{2} satisfies the integral condition \eqref{cond. de con.2}?
\vspace{0.2cm}

Naturally, another point to be considered is the following one.

\vspace{0.2cm}
\noindent\textbf{Problem $\mathcal{B}$:}  What assumptions are needed to ensure that the solution $u$ of \eqref{2} is unique and verifies \eqref{cond. de con.2} for a unique $f_0$?
\vspace{0.2cm}

Finally, with these results in hand, the last problem of this article is related with the existence of a minimal time for which the integral overdetermination condition \eqref{cond. de con.2} be satisfied. Precisely, the problem can be seen as follows.

\vspace{0.2cm}
\noindent\textbf{Problem $\mathcal{C}$:}  Can on find a time $T_0>0$, depending of the boundary and initial data, such that if $T \leq T_0$, there exist a function $f_0$, in appropriated space, in that way that the solution $u$ of \eqref{2} verifies \eqref{cond. de con.2}?
\vspace{0.2cm}

In summary, the main goal of this manuscript is to prove that these problems are indeed true. There are basically some features to be emphasized.
\begin{itemize}
\item The integral overdetermination condition is effective and gives good control properties. This kind of condition was first applied in the inverse problem (see e.g. \cite{PriOrVas}) and, more recently, in control theory \cite{CaSo,Fa1,FA}.
\vspace{0.1cm}
\item One should be capable of controlling the system, when the control acts in $[0,T]$, on an unbounded domain, which is new for the Kawahara equation.
\vspace{0.1cm}
\item We are also able to prove the existence of a minimal $T>0$ such that the overdetermination condition is still verified, however, we believe that this time is not optimal.
\end{itemize}

\subsection{Main results} In this paper we are able to present answers to the problems $\mathcal{A}$ and $\mathcal{B}$ that were firstly proposed in \cite{CaGo}. Additionally, the results of this work extend the results presented in \cite{CaGo} for a new framework for Kawahara equation, that is: The real line, right half-line and left half-line.  For sake of simplicity, we will present here the \textit{overdetermination control problem} in the right half-line, for details of the results for the real line and left half-line we invite the reader to read Section \ref{Sec4} at the end of this article.

In this way, the first result ensures that the \textit{overdetermination control problem}, that is, the internal control problem with an integral condition like \eqref{cond. de con.2} on unbounded domain follows for small data, giving answers for the Problem $\mathcal{A}$ and $\mathcal{B}$.

\begin{theorem}\label{main1} Let $T > 0$ and $p \in [2,\infty]$. Consider $\mu \in H^{\frac{2}{5}}(0,T) \cap L^{p}(0,T)$, $\nu \in H^{\frac{1}{5}}(0,T) \cap L^{p}(0,T)$, $u_{0} \in L^{2}(\mathbb{R}^{+})$ and $\varphi \in W^{1,p}(0,T)$. Additionally, let  $g \in C(0,T; L^{2}(\mathbb{R}^{+}))$ and  $\omega$ be a fixed function which belongs to the following set 
\begin{equation}\label{jota}
\mathcal{J}= \{\omega \in H^{5}(\mathbb{R}^+): \omega(0)= \omega'(0)= \omega''(0)=0 \},
\end{equation}
 satisfying
$$
\varphi(0)= \int_{\mathbb{R}^{+}}^{}u_{0}(x)\omega(x)dx
$$
and
$$
\biggl|\int_{\mathbb{R}^{+}}^{}g(t,x)\omega(x)dx\biggr| \geq g_{0} > 0, \ \forall t \in [0,T],
$$
where $g_{0}$ is a constant. Then, for each $T > 0$ fixed, there exists a constant $\gamma > 0$ such that if 
$$c_{1}= \|u_{0}\|_{L^{2}(\mathbb{R}^{+})} + \|\mu\|_{H^{\frac{2}{5}}(0,T)} +  \|\nu\|_{H^{\frac{1}{5}}(0,T)} + \|\varphi'\|_{L^{2}(0,T)} \leq \gamma,$$ we can find a unique control input $f_{0} \in L^{p}(0,T)$ and a unique solution $u$ of  \eqref{2} satisfying  \eqref{cond. de con.2}.
\end{theorem}

Our second result gives us a small time interval for which the integral overdetermination condition \eqref{cond. de con.2} follows for solutions of \eqref{2}. Precisely, the answer for the Problem $\mathcal{C}$ can be read as follows. 

\begin{theorem}\label{main2} Suppose the hypothesis of Theorem \ref{main1} be satisfied and consider $\delta:=T^{\frac{1}{5}}\in (0,1)$, for $T>0$. Then there exists $T_{0}:=\delta_0^{\frac{1}{5}}>0$, depending on $c_1=c_1(\delta)$ given by
$$c_1(\delta):=\|u_{0\delta}\|_{L^{2}(\mathbb{R}^{+})} +  \|\varphi'_{\delta}\|_{L^{2}(0,T)} + \|\mu_{\delta}\|_{H^{\frac{2}{5}}(0,T)} +  \|\nu_{\delta}\|_{H^{\frac{1}{5}}(0,T)},$$
such that if  $T \leq T_0$, there exist a control function $f_{0} \in L^{p}(0, T)$ and a solution $u$ of \eqref{2} verifying \eqref{cond. de con.2}.
\end{theorem}

From the previous results, we are able to give a consequence related to the controllability of the following system 
\begin{equation}\label{2_1}
\left\lbrace
\begin{array}{llr}
u_{t} + \alpha u_{x} + \beta u_{xxx} - u_{xxxxx} + uu_{x}= f_0(t)g(t,x)& \mbox{in} \ [0,T]\times\mathbb{R}^{+}, \\ 
u(t,0)= u_{x}(t,0)= 0 & \mbox{on} \ [0,T], \\
 u(0,x) = u_{0}(x) & \mbox{in} \ \mathbb{R}^{+},
\end{array}\right. 
\end{equation} 
posed in the right half-line.  Precisely,  we present a control property involving the overdetermination condition \eqref{cond. de con.2} and the initial state $u_0$ and final state $u_T$.  To do that, consider the following notation
\begin{equation}\label{mass}
[u(x,t)]=\int_{\R^+} u(x,t)d\eta,
\end{equation}
which one will be called of mass,  for some $\sigma$-finite measure  $\eta$  in $\R^+$.  With this in hand,  as a consequence of Theorem \ref{main1}, the following exact controllability in the right half-line holds true.
\begin{corollary}\label{teocontrol}
Let $T > 0$ and $p \in [2,\infty]$. Consider  $u_{0}, \ u_T  \in L^{2}(\mathbb{R}^{+})$ and  $g \in C(0,T; L^{2}(\mathbb{R}^{+}))$, 
 satisfying
\begin{equation}\label{damping}
\biggl|\int_{\mathbb{R}^{+}}^{}g(t,x)dx\biggr| \geq g_{0} > 0, \ \forall t \in [0,T],
\end{equation}
where $g_{0}$ is a constant.  Additionally,  consider $\omega$ be a fixed function which belongs to the following set \eqref{jota}
and $\varphi \in W^{1,p}(0,T)$  satisfying
\begin{equation} \label{eqq1}
\varphi(0)= \int_{\R^+} u_0(x)\omega(x)dx \quad \text{and} \quad \varphi(T)= \int_{\R^+} u_T(x)\omega(x)dx. 
\end{equation}  Then, for each $T > 0$ fixed, there exists a constant $\gamma > 0$ such that if 
$$ \|u_{0}\|_{L^{2}(\mathbb{R}^{+})} +  \|\varphi'\|_{L^{2}(0,T)} \leq \gamma,$$ we can find a unique control input $f_{0} \in L^{p}(0,T)$, a unique solution $u$ of  \eqref{2_1} and a $\sigma$-finite measure $\eta$ in $\R^+$ such that 
\begin{equation}\label{exat}
[u(T)]=[u_T].
\end{equation}
\end{corollary}

\subsection{Historical background}  Is well know that Cauchy problem and control theory for the Kawahara equation 
\begin{equation}\label{int}
u_{t} + \alpha u_{x} + \beta u_{xxx} - u_{xxxxx} + uu_{x}=0
\end{equation}
has been studied by several mathematicians in recent years in differents framework: bounded domain of $\mathbb{R}$, on real line $\mathbb{R}$, on the torus $\mathbb{T}$, right half-line $\mathbb{R}^+$ and left half-line $\mathbb{R}^-$.

With respect to the well-posedness of the Kawahara equation, the first local result is due to Cui and Tao \cite{CuiTao}. The authors proved a Strichartz estimate for the fifth-order operator and obtained the local well-posedness in $H^{s}(\mathbb{R})$, for  $s>1 / 4$. After that,  Cui \textit{et al.} \cite{CuiDeTao} improved the previous result to the negative regularity Sobolev space $H^{s}(\mathbb{R}), s>-1$. Is important to point out that Wang \textit{et al.} \cite{WaCuiDen} improved to a lower regularity, in this case, $s \geq-7 / 5 .$  These papers treated the problem using Fourier restriction norm method.  In \cite{ChenLiMiaoWu} and \cite{JiaHuo}, authors showed the local well-posedness in $H^{s}(\mathbb{R}), s>-7 / 4$, while their methods are same, particularly, the Fourier restriction norm method in addition to Tao's $[K ; Z]$-multiplier norm method. At the critical regularity Sobolev space $H^{-7 / 4}(\mathbb{R})$, Chen and Guo \cite{ChenGuo} proved local and global well-posedness by using Besov-type critical space and I-method. Kato \cite{Kato} studied local wellposedness for $s \geq-2$ by modifying $X^{s, b}$ space and the ill-posedness for $s<-2$ in the sense that the flow map is discontinuous.  

Finally, still regarding the well-posedness results, we refer to two recent works that treat the Kawahara equation. Recently, Cavalcante and Kwak \cite{MC} studied the IBVP of the Kawahara equation posed on the right and left half-lines with the nonlinearity as in \eqref{int}. Being precise, they proved the local well-posedness in the low regularity Sobolev space, that is, $s \in\left(-\frac{7}{4}, \frac{5}{2}\right) \backslash\left\{\frac{1}{2}, \frac{3}{2}\right\}$. Additionally, the authors in \cite{MC1} extended the argument of \cite{MC} to fifth-order KdV-type equations with different nonlinearities, in specific, where the scaling argument does not hold. They are established in some range of $s$ where the local well-posedness of the IBVP fifth-order KdV-type equations on the right half-line and the left half-line holds true.

Stabilization and control problems (see \cite{zhang1,CaKawahara} for details of these kinds of issues) has been studied in recent years for the Kawahara Equation, however with few results in the literature. A first work concerning to the stabilization property for the Kawahara equation in a bounded domain $Q_T = (0, T) \times (0,L)$,
\begin{equation}\label{int1}
\left\lbrace
\begin{array}{llr}
u_{t} + u_{x} +u_{xxx} - u_{xxxxx}+uu_x= f(t,x) & \mbox{in} \ Q_{T}, \\ 
u(t,0)=h_{1}(t), \ u(t,L)= h_{2}(t), \ u_{x}(t,0)= h_{3}(t)& \mbox{on} \ [0,T], \\
u_{x}(t,L)=h_{4}(t), \ u_{xx}(t,L) = h(t) & \mbox{on} \ [0,T],\\
 u(0,x) = u_{0}(x) & \mbox{in} \ [0,L], 
\end{array}\right. 
\end{equation} 
is due to Capistrano-Filho \textit{et al.} in \cite{CaKawahara}. In this paper the authors were able to introduce an internal feedback law in \eqref{int1}, considering general nonlinearity $u^pu_x$, $p\	in[1,4)$, instead of $uu_x$, and $h(t)=h_i(t)=0$, for $i=1,2,3,4$. Being precise, they proved that under the effect of the damping mechanism the energy associated with the solutions of the system decays exponentially.


Now, some references of internal control problems are presented. This problem was first addressed in \cite{zhang} and after that in \cite{zhang1}. In both cases the authors considered the Kawahara equation in a periodic domain $\mathbb{T}$ with a distributed control of the form \[ f(t,x)=(Gh)(t,x):= g(x)(  h(t,x)-\int_{\mathbb{T}}g(y) h(t,y) dy), \] where $g\in C^\infty (\mathbb T)$ supported in $\omega\subset\mathbb{T}$ and $h$ is a control input. Here, it is important to observe that the control in consideration has a different form as presented in \eqref{f}, and the result is proven in a different direction from what we will present in this manuscript.

Still related with internal control issues, Chen \cite{MoChen} presented results considering the Kawahara equation \eqref{int1} posed on a bounded interval with a distributed control $f(t,x)$ and homogeneous boundary conditions. She showed the result taking advantage of a Carleman estimate associated with the linear operator of the Kawahara equation with an internal observation. With this in hand, she was able to get a null controllable result when $f$ is effective in a $\omega\subset(0,L)$.  As the results obtained by her do not answer all the issues of the internal controllability, in a recent article \cite{CaGo} the authors closed some gaps left in \cite{MoChen}. Precisely, considering the system \eqref{int1} with an internal control $f(t,x)$ and homogeneous boundary conditions, the authors are able to show that the equation in consideration is exact controllable in $L^2$-weighted Sobolev spaces and, additionally, the Kawahara equation is controllable by regions on $L^2$-Sobolev space, for details see \cite{CaGo}.

Finally, with respect to a new tool to find control properties for dispersive systems, we can cite a recent work of the first two authors \cite{CaSo}. In this work, the authors showed a new type of controllability for a dispersive fifth order equation that models water waves, what they called \textit{overdetermination control problem}. Precisely, they are able  to find a control acting at the boundary that guarantees that the solution of the problem under consideration satisfies an integral overdetermination condition. In addition, when the control acts internally in the system, instead of the boundary, the authors proved that this condition is satisfied. These problems give answers that were left open in \cite{CaGo} and present a new way to prove boundary and internal controllability results for a fifth order KdV type equation.

\subsection{Heuristic and outline of the article} The goal of this manuscript is to investigate and discuss control problems with an integral condition on an unbounded domain. Precisely, we study the internal control problem when the solution of the system satisfies \eqref{cond. de con.2}, so we intend to extend - for unbounded domains - a new way to prove internal control results for the system \eqref{int1}, initially proposed in \cite{Fa1,FA}, for KdV equation, and more recently in \cite{CaSo}, for Kawahara equation in a bounded domain. Thus, for this type of \textit{integral overdetermination condition} the first results on the solvability of control problems for the IBVP of Kawahara equation on unbounded domains are obtained in the present paper.

The first result, Theorem \ref{main1}, is concerning the \textit{internal overdetermination control problem}. Roughly speaking, we are able to find an appropriate control $f_0$, acting on $[0,T]$ such that integral condition \eqref{cond. de con.2} it turns out. First, we borrowed the existence of solutions for the IBVP \eqref{2} of \cite{MC}. With these results in hand, for the special case when $s=0$,  Theorem \ref{main1} is first proved for the linear system associated to \eqref{2} and after that, using a fixed point argument, extended to the nonlinear system. The main ingredients are auxiliary lemmas presented in the Section \ref{Sec2}. In one of these lemmas (see Lemma \ref{lss} below)  we are able to find two appropriate applications that link the internal control term $f_0(t)$ with the overdetermination condition \eqref{cond. de con.2}, namely
\begin{equation*}
\begin{split}
\Lambda:  L^{p}(0,T)&\longrightarrow \widetilde{W}^{1,p}(0,T)\\
\quad \quad  \quad f_0 &\longmapsto (\Lambda f_0)(\cdot)= \int_{\mathbb{R}^+}u(\cdot,x)\omega(x)dx
\end{split}
\end{equation*}
and
\begin{equation*}
\begin{array}{lll}
A: L^{p}(0,T) &\longrightarrow L^{p}(0,T)\\
\quad \quad \quad f_0 &\longmapsto \displaystyle (Af_{0})(\cdot)= \frac{\varphi'(\cdot)}{g_{1}(\cdot)} - \frac{1}{g_{1}(\cdot)}\int_{\mathbb{R}^{+}}^{}u(t,x)(\alpha\omega' + \beta\omega''' - \omega''''')dx,
\end{array}
\end{equation*}
where, $$ g_{1}(\cdot)= \int_{0}^{L}g(\cdot,x)\omega(x)dx.$$ So, we  prove that such application has an inverse which is continuous, by Banach's theorem, showing the lemma in question, and so, reaching our goal, to prove Theorem \ref{main1}.

With the previous result in hand, the answer for the Problem $\mathcal{C}$ is given by  Theorem \ref{main2}. This result gives us a minimal time which the integral condition \eqref{cond. de con.2} is satisfied. To be more precise, Theorem \ref{main2} is proved in three parts. First part, we give a refinement of Lemma \ref{lss}, namely, Lemma \ref{lss1}. With this in hand, we need, in a second moment, to use the scaling of our equation \eqref{2} to produce a ``new" Kawahara equation on $Q^+_T$ . This gives us the possibility to use the Theorem \ref{main1} and, with help of Lemma \ref{lss1}, reach the proof of Theorem \ref{main2}.

Finally, as a consequence of Theorem \ref{main1}, we produce a type of exact controllability result (Corollary \ref{teocontrol}). More precisely, we show that the mass of the system \eqref{mass} is reached on the final time $T$,  that is,  \eqref{exat} holds.

Thus, we finish our introduction showing the structure of the manuscript.  Section \ref{Sec1} is devoted to presenting some preliminaries, which are used throughout the article. Precisely, we present the Fourier restriction spaces related with the operator of the Kawahara, moreover, reviewed the main results of the well-posedness for the fifth order KdV equation in these spaces. In the Section \ref{Sec2} we present some auxiliary lemmas which help us to prove the internal controllability results. The overdetermination control results, when the control is acting internally, is presented in the Section \ref{Sec3}, that is, we present the proof of the main results of the manuscript, Theorems \ref{main1}, \ref{main2} and Corollary \ref{teocontrol}. Finally, in the Section \ref{Sec4} we present some further comments and some conclusions of the generality of the work.

\section{Preliminaries}\label{Sec1}
\subsection{Fourier restriction spaces}  Let $f$ be a Schwartz function, i.e., $f \in \mathcal{S}_{t,x}(\R \times \T)$, $\wt{f}$ or $\ft (f)$ denotes the space-time Fourier transform of $f$ defined by
\[\wt{f}(\tau,\xi)=\frac{1}{2\pi}\int_{\mathbb{R}^2}e^{-ix\xi}e^{-it\tau}f(t,x)\; dxdt .\]
Moreover, we use $\ft_x$ (or $\wh{\;}$ ) and $\ft_t$ to denote the spatial and temporal Fourier transform, respectively.

For given $s,b\in\mathbb{R}$, we define the space $X^{s,b}$ associated to \eqref{2} as the closure of $\mathcal{S}_{t,x}(\R \times \T)$ under the norm
$$
\|f\|_{X^{s, b}}^{2}=\int_{\mathbb{R}^{2}}\langle\xi\rangle^{2 s}\left\langle\tau-\xi^{5}\right\rangle^{2 b}|\widetilde{f}(\tau, \xi)|^{2} d \xi d \tau
$$
where  $\langle\cdot\rangle=(1+|\cdot|^2)^{1/2}$.

As well-known, the $X^{s, b}$ space with $b>\frac{1}{2}$ is well-adapted to study the IVP of dispersive equations. The function space equipped with the Fourier restriction norm, which is the so-called $X^{s,b}$ spaces, has been proposed by Bourgain \cite{Bourgain1993-1,Bourgain1993-2} to solve the periodic NLS and generalized KdV. Since then, it has played a crucial role in the theory of dispersive equations, and has been further developed by many researchers, in particular,  Kenig, Ponce and Vega \cite{KPV1996} and Tao \cite{Tao2001}. 

In our case, to study the IBVP \eqref{2}, is requested us to introduce modified $X^{s, b}$-type spaces. So, we define the (time-adapted) Bourgain space $Y^{s, b}$ associated to \eqref{2} as the completion of $\mathcal{S}\left(\mathbb{R}^{2}\right)$ under the norm
$$
\|f\|_{Y^{s, b}}^{2}=\int_{\mathbb{R}^{2}}\langle\tau\rangle^{\frac{2 s}{5}}\left\langle\tau-\xi^{5}\right\rangle^{2 b}|\widetilde{f}(\tau, \xi)|^{2} d \xi d \tau.
$$
Additionally, due the study of the of the IBVP introduced in \cite{MC}, they used the low frequency localized $X^{0, b}$-type space with $b>\frac{1}{2}$ in the nonlinear estimates. Hence, we need also define $D^{\alpha}$ space as the completion of $\mathcal{S}\left(\mathbb{R}^{2}\right)$ under the norm
$$
\|f\|_{D^{\alpha}}^{2}=\int_{\mathbb{R}^{2}}\langle\tau\rangle^{2 \alpha} 1_{\{\xi:|\xi| \leq 1\}}(\xi)|\widetilde{f}(\tau, \xi)|^{2} d \xi d \tau
$$
where $1_{A}$ is the characteristic functions on a set $A$. With this in hand, now we set the solution space denoted by $Z_{1}^{s, b, \alpha}$ with the following norm
$$
\|f\|_{Z_{1}^{s, b, \alpha}\left(\mathbb{R}^{2}\right)}=\sup _{t \in \mathbb{R}}\|f(t, \cdot)\|_{H^{s}}+\sum_{j=0}^{1} \sup _{x \in \mathbb{R}}\left\|\partial_{x}^{j} f(\cdot, x)\right\|_{H^{\frac{s+2-j}{5}}}+\|f\|_{X^{s, b} \cap D^{\alpha}}.
$$
The spatial and time restricted space of $Z_{1}^{s, b, \alpha}\left(\mathbb{R}^{2}\right)$ is defined by the standard way:
$$
Z_{1}^{s, b, \alpha}\left((0, T) \times \mathbb{R}^{+}\right)=\left.Z_{1}^{s, b, \alpha}\right|_{(0, T) \times \mathbb{R}^{+}}
$$
equipped with the norm
$$
\|f\|_{Z_{1}^{g, b, \alpha}\left((0, T) \times \mathbb{R}^{+}\right)}=\inf _{g \in Z_{1}^{s, b, \alpha}}\left\{\|g\|_{Z_{1}^{s, b, \alpha}}: g(t, x)=f(t, x) \text { on }(0, T) \times \mathbb{R}^{+}\right\}.
$$

\subsection{Overview of the well-posedness results} In this section we are interested to present the well-posedness results for the Kawahara system, namely, 
\begin{equation}\label{pro2}
\left\lbrace
\begin{array}{llr}
u_{t} + \alpha u_{x} + \beta u_{xxx} - u_{xxxxx} = f(t,x) & \mbox{em} \ [0,T]\times\mathbb{R}^{+}, \\ 
u(t,0)=\mu(t), \ u_{x}(t,0)= \nu(t) & \mbox{em} \ [0,T], \\
 u(0,x) = u_{0}(x) & \mbox{em} \ \mathbb{R}^{+}.
\end{array}\right. 
\end{equation} 
The results presented here are borrowed from \cite{MC}  and give us good properties of the IBVP \eqref{pro2}. The first one give a relation of the nonlinearity involved in our problem with the Fourier restriction spaces introduce in the previous subsection. Precisely, we have the nonlinear term $f=uu_x$ can be controlled in the  $X^{s,-b}$ norm.
\begin{proposition}\label{pdxv} 
For $-7 / 4<s$, there exists $b=b(s)<1 / 2$ such that for all $\alpha>1 / 2$, we have
\begin{equation}\label{dxv}
\left\|\partial_{x}(u v)\right\|_{X^{s,-b}} \leq c\|u\|_{X^{s, b} \cap D^{\alpha}}\|v\|_{X^{s, b} \cap D^{\alpha}}.
\end{equation}
\end{proposition}
\begin{proof}
See \cite[Proposition 5.1]{MC}.
\end{proof}

Now on, we will consider the following: $s= 0$, $b(s)= b_{0}$,  $\alpha(s)= \alpha_{0}$ and $ Z_{1}^{0,b_{0}, \alpha_{0}}(Q^{+}_{T})= Z(Q^{+}_{T})$.  As a consequence of the previous proposition, we have the following. 

\begin{corollary}\label{pdxv2} 
There exists $b_{0} \in (0,\frac{1}{2})$  such that for all $\alpha_{0} > \frac{1}{2},$ follows that 
\begin{equation}\label{dxv2}
\|\partial_{x}(uv)\|_{X^{0,-b_{0}}(Q^{+}_{T})} \leq C\|u\|_{Z(Q^{+}_{T})}\|v\|_{Z(Q^{+}_{T})},
\end{equation}
for any $u, v \in Z(Q^{+}_{T}).$
\end{corollary}

Now, we are interested for a special case of the well-posedness result presented in  \cite{MC}. Be precise, considering $s=0$, \cite[Theorem 1.1]{MC} gives us the following result.
\begin{theorem}\label{BC1} Let $T > 0$ and $u_{0} \in L^{2}(\mathbb{R}^{+})$, $\mu \in H^{\frac{2}{5}}(0,T)$, $\nu \in H^{\frac{1}{5}}(0,T)$ and $ f \in X^{0,-b_{0}}(Q^{+}_{T}),$ for $b_{0} \in (0, \frac{1}{2})$. Then there exists a unique solution $u := S(u_{0}, \mu, \nu, f) \in Z(Q^{+}_{T})$ of \eqref{pro2} such that
\begin{equation}\label{norma de u1}
\begin{array}{lll}
\|u\|_{Z(Q_{T}^{+})} \leq C_{0}\left(\|u_{0}\|_{L^{2}(\mathbb{R}^{+})} + \|\mu\|_{H^{\frac{2}{5}}(0,T)}+\|\nu\|_{H^{\frac{1}{5}}(0,T)} + \|f\|_{X^{0,-b_{0}}(Q_{T}^{+})} \right)
\end{array}
\end{equation} 
where $C_{0} > 0$ is a positive constant depending only of $b_{0}, \alpha_{0}$ and $T$.
\end{theorem}

\section{Key lemmas}\label{Sec2}

In this section we are interested to prove some auxiliary lemmas for the solutions of the  system
\begin{equation}\label{Al}
\left\lbrace
\begin{array}{llr}
u_{t} + \alpha u_{x} + \beta u_{xxx} - u_{xxxxx} = f(t,x) & \mbox{in} \ [0,T]\times\mathbb{R}^{+}, \\ 
u(t,0)=\mu(t), \ u_{x}(t,0)= \nu(t) & \mbox{on} \ [0,T], \\
 u(0,x) = u_{0}(x) & \mbox{in} \ \mathbb{R}^{+}.
\end{array}\right. 
\end{equation} 
These lemmas will be the key to proof the main results of this work. 

To do this, consider $\omega\in\mathcal{J}$ defined by \eqref{jota} and define  $q:[0,T] \longrightarrow \mathbb{R}$ as follows
\begin{equation}\label{q}
q(t)= \int_{\mathbb{R}^{+}}u(t,x)\omega(x)dx,
\end{equation}
where $u:= S(u_{0}, \mu, \nu, f_{1} + f_{2x})$  is solution of \eqref{Al} guaranteed by Theorem \ref{BC1}. The next two auxiliary lemmas are the key point to show the main results of this work. The first one, gives that $q\in W^{1,p}(0,T)$ and can be read as follows.
\begin{lemma}\label{1.1} Let  $T > 0$, $p \in [2, \infty]$ and the assumptions of Theorem \ref{BC1} be satisfied,  with $\ f= f_{1} + f_{2x}$, where $f_{1} \in L^{p}(0,T; L^{2}(\mathbb{R}^{+}))$,  $f_{2} \in L^{p}(0,T; L^ {1}(\mathbb{R}^{+}))$ and $\mu, \nu \in L^{p}(0,T)$. If $u \in Z(Q^{+}_{T})$ is a solution of \eqref{pro2} and $\omega \in \mathcal{J}$, defined in \eqref{jota}, then the function $q\in W^{1,p}(0,T)$ and the relation
\begin{equation}\label{q'}
\begin{split}
q'(t) =& \ \omega'''(0)\nu(t) - \omega''''(0)\mu(t) + \int_{\mathbb{R}^{+}}^{}f_{1}(t,x)\omega(x)dx- \int_{\mathbb{R}^{+}}^{}f_{2}(t,x)\omega'(x)dx\\
&+ \int_{\mathbb{R}^{+}}^{}u(t,x)[\alpha\omega'(x) + \beta\omega'''(x) - \omega'''''(x)]dx  
\end{split}
\end{equation} 
holds for almost all $t\in[0,T]$. In addition, the function $q'\in L^{p}(0,T)$ can be estimate in the following way
\begin{equation}\label{norma de q'}
\begin{split}
\|q'\|_{L^{p}(0,T)} \leq& \ C\left(( \|u_{0}\|_{L^{2}(\mathbb{R}^{+})} + \|\mu\|_{(L^{p}\cap H^{\frac{2}{5}})(0,T)} + \|\nu\|_{(L^{p}\cap H^{\frac{1}{5}})(0,T)} \right.\\
&+\left.\|f_{1}\|_{L^{p}(0,T;L^{2}(\mathbb{R}^{+}))} +  \|f_{2}\|_{L^{p}(0,T;L^{1}(\mathbb{R}^{+}))} + \|f_{2x}\|_{X^{0,-b_{0}}(Q^{+}_{T})} \right)
\end{split}
\end{equation}
with $C= C(|\alpha|,|\beta|, T, \|\omega\|_{\mathbb{R}^{+}}) > 0$ a constant that is nondecreasing with increasing $T$.
\end{lemma}

\begin{proof} Considering $\psi \in C_{0}^{\infty}(0,T)$, multiplying \eqref{Al} by $\psi\omega$ and inegrating by parts in $[0,T]\times[0,R]$, for some $R > 0$, we get, using the boundary condition of \eqref{Al} and the hypothesis that $\omega\in\mathcal{J}$, that
\begin{equation*}
\begin{split}
-\int_{0}^{T}\psi'(t)q(t)dt=&
\int_{0}^{T}\int_{\mathbb{R}^{+}}u_{t}(t,x)\psi(t)\omega(x)dxdt \\
=& \int_{0}^{T}\psi(t)\left(\int_{\mathbb{R}^{+}}^{}u(t,x)(\alpha\omega'(x) + \beta\omega'''(x) - \omega'''''(x))dx\right.\\
&+ \int_{\mathbb{R}^{+}}f_{1}(t,x)\omega(x)dx - \int_{\mathbb{R}^{+}}^{}f_{2}(t,x)\omega'(x)dx\\
&-\omega''''(0)\mu(t) + \omega'''(0)\nu(t)\biggr)dt\\
=&\int_0^T\psi(t)r(t)dt,
\end{split}
\end{equation*} 
with $ r: [0,T] \longmapsto \mathbb{R}$ defined by
\begin{equation*}
\begin{split}
r(t)=& \int_{\mathbb{R}^{+}}u(t,x)(\alpha\omega'(x) + \beta\omega'''(x) - \omega'''''(x))dx  -\omega''''(0)\mu(t) + \omega'''(0)\nu(t)\\
&+  \int_{\mathbb{R}^{+}}f_{1}(t,x)\omega(x)dx - \int_{\mathbb{R}^{+}}^{}f_{2}(t,x)\omega'(x)dx\\
:=&I_1+I_2+I_3,
\end{split}
\end{equation*} 
which gives us $q'(t)=r(t)$, where
\begin{equation*}
\begin{split}
I_1=& \int_{\mathbb{R}^{+}}u(t,x)(\alpha\omega'(x) + \beta\omega'''(x) - \omega'''''(x))dx  -\omega''''(0)\mu(t) + \omega'''(0)\nu(t),\\
I_2=& - \int_{\mathbb{R}^{+}}^{}f_{2}(t,x)\omega'(x)dx, \\
I_3=&  \int_{\mathbb{R}^{+}}f_{1}(t,x)\omega(x)dx .
\end{split}
\end{equation*}

It remains for us to prove that $q' \in L^{p}(0,T)$, for $p\in[2,\infty]$. To do it, we need to bound each term of \eqref{q'}. We will split this analysis in two steps. 

\vspace{0.2cm}
\noindent\textbf{Step 1.} $2 \leq p < \infty$
\vspace{0.2cm}

Let us first to bound $I_1$. To do this, note that, for $t \in [0,T]$, we have
\begin{multline*}
 \biggl| \int_{\mathbb{R}^{+}}u(t,x)(\alpha\omega'(x)+ \beta\omega'''(x) - \omega'''''(x))dx\biggr| \\
 \leq (|\alpha|\|\omega'\|_{L^{2}(\mathbb{R}^{+})} + |\beta|\|\omega'''\|_{L^{2}(\mathbb{R}^{+})} + \|\omega'''''\|_{L^{2}(\mathbb{R}^{+})})\|u(t,\cdot)\|_{L^{2}(\mathbb{R}^{+})}.
\end{multline*}
Moreover, the trace terms are bounded thanks to the fact that $\omega\in\mathcal{J}$. Thus, this yields that
\begin{equation*}
\begin{split}
 \biggl\| \int_{\mathbb{R}^{+}}u(t,x)(\alpha\omega(x)& +\beta\omega'''(x) - \omega'''''(x))dx\biggr\|_{L^{p}(0,T)}\leq C(|\alpha|, |\beta|, \|\omega\|_{H^{5}(\mathbb{R}^{+})})\|u\|_{L^{p}(0,T;L^{2}(\mathbb{R}^{+}))}.
\end{split}
\end{equation*}
Since
$$
\|u\|_{L^{p}(0,T; L^{2}(\mathbb{R}^{+}))} \leq  T^{\frac{1}{p}}\|u\|_{C(0,T; L^{2}(\mathbb{R}^{+}))},
$$
we have that 
$$\biggl\| \int_{\mathbb{R}^{+}}u(t,x)(\alpha\omega(x)+ \beta\omega'''(x) - \omega'''''(x))dx\biggr\|_{L^{p}(0,T)} \leq C(|\alpha|, |\beta|, \|\omega\|_{H^{5}(\mathbb{R}^{+})})T^{\frac{1}{p}}\|u\|_{C(0,T; L^{2}(\mathbb{R}^{+}))}.
$$
Now, let us estimate $I_2$. For this case we start observing that
\begin{equation*}
\begin{split}
\biggl| \int_{\mathbb{R}^{+}}f_{2}(t,x)\omega'(x)dx\biggr|\leq& \int_{\mathbb{R}^{+}}^{}|f_{2}(t,x)\omega'(x)|dx \\
\leq& \|\omega'\|_{C(\mathbb{R}^{+})}\|f_{2}(t,\cdot)\|_{L^{1}(\mathbb{R}^{+})}\\
\leq& \ C\|\omega'\|_{H^{1}(\mathbb{R}^{+})}\|f_{2}(t,\cdot)\|_{L^{1}(\mathbb{R}^{+})}\\
\leq & \ C\|\omega\|_{H^{5}(\mathbb{R}^{+})}\|f_{2}(t,\cdot)\|_{L^{1}(\mathbb{R}^{+})},
\end{split}
\end{equation*}
where we have used the following continuous embedding
$$H^{1}(\mathbb{R}^{+}) \hookrightarrow (L^{\infty}(\mathbb{R}^{+}) \cap C(\mathbb{R}^{+})).$$
Therefore, we get that 
$$
\biggl\| \int_{\mathbb{R}^{+}}^{}f_{2}(t,x)\omega'(x)dx\biggr\|_{L^{p}(0,T)} \leq C(\|\omega\|_{H^{5}(\mathbb{R}^{+})})\|f_{2}\|_{L^{p}(0,T; L^{1}(\mathbb{R}^{+}))}.
$$
In a similar way, we can bound $I_3$ as
$$
\biggl\| \int_{\mathbb{R}^{+}}^{}f_{1}(t,x)\omega(x)dx\biggr\|_{L^{p}(0,T)} \leq \|\omega\|_{L^{2}(\mathbb{R}^{+})}\|f_{1}\|_{L^{p}(0,T;L^{2}(\mathbb{R}^{+}))}.
$$

With these estimates in hand and using the hypothesis over $ \mu$ and $\nu$, that is,  $ \mu$ and $\nu$ belonging to $L^{p}(0,T)$, we have $r \in L^{p}(0,T)$, which implies that $q \in W^{1,p}(0,T)$ and 
\begin{equation*}
\begin{split}
\|q'\|_{L^{p}(0,T)} \leq& \widetilde{C}(|\alpha|, |\beta|, T, \|\omega\|_{H^{5}(\mathbb{R}^{+})})\biggl( \|\mu\|_{L^{p}(0,T)} + \|\nu\|_{L^{p}(0,T)} + \|u\|_{Z(Q^{+}_{T})}\\
&+ \|f_{1}\|_{L^{p}(0,T;L^{2}(\mathbb{R}^{+}))} + \|f_{2}\|_{L^{p}(0,T;L^{1}(\mathbb{R}^{+}))} \biggr).
\end{split}
\end{equation*}
Finally, using \eqref{norma de u1} in the previous inequality, \eqref{norma de q'} holds.

\vspace{0.2cm}
\noindent\textbf{Step 2.} $p= \infty$
\vspace{0.2cm}

Observe that thank to the relation \eqref{q'} and the fact that 
$$H^{1}(\mathbb{R}^{+}) \hookrightarrow (L^{\infty}(\mathbb{R}^{+}) \cap C(\mathbb{R}^{+}),$$ 
we get that
\begin{equation*}
\begin{split} 
|q'(t)| \leq & (|\alpha|\|\omega'\|_{L^{2}(\mathbb{R}^{+})} +|\beta|\|\omega'''\|_{L^{2}(\mathbb{R}^{+})} + \|\omega'''''\|_{L^{2}(\mathbb{R}^{+})})\|u(t,\cdot)\|_{L^{2}(\mathbb{R}^{+})} \\&+ \|\omega\|_{L^{2}(\mathbb{R}^{+})} \|f_{1}(t,\cdot)\|_{L^{2}(\mathbb{R}^{+})} 
 + \|\omega'\|_{H^{1}(\mathbb{R}^{+})}\|f_{2}(t,\cdot)\|_{L^{1}(\mathbb{R}^{+})} \\
 &+ |\omega''''(0)||\mu(t)| + |\omega'''(0)||\nu(t)|.
\end{split}
\end{equation*}
Thus, 
\begin{equation*}
\begin{split} 
 \|q'\|_{C(0,T)} \leq
C\left( \|u\|_{Z_{1}(Q_{+}^{T}))} + \|f_{2}\|_{C(0,T;L^{1}(\mathbb{R}^{+}))} + \|f_{1}\|_{C(0,T;L^{2}(\mathbb{R}^{+}))} + \|\mu\|_{C(0,T)} + \|\nu\|_{C(0,T)} \right),
\end{split}
\end{equation*}
with $C= C(|\alpha|,|\beta|, \|\omega\|_{H^{5}(\mathbb{R}^{+})}, |\omega''''(0)|, |\omega'''(0)|) > 0$. Thus, Step 2 is achieved using \eqref{norma de u1} and the proof of the lemma is complete.
\end{proof}

\begin{remarks}We will give some remarks in order related with the previous lemma. 
\begin{itemize}
\item[i.] We are implicitly assuming that $f_{2x}\in L^1(0, T; L^2(\mathbb{R}^+))$ but it is not a problem, since the function that we will take for $f_2$, in our purposes, satisfies that condition.
\item[ii.] When $p =\infty$ the spaces $L^p(0, T)$, $L^p(0, T; L^2(\mathbb{R}^+))$ and $L^p(0, T; L^1(\mathbb{R}^+))$ are replaced by the spaces $C([0, T])$, $C([0, T]; L^2(\mathbb{R}^+))$ and $C([0, T]; L^1(\mathbb{R}^+))$, respectively. So, we can obtain $q\in C^1([0, T])$.
\end{itemize}
\end{remarks}

Now, consider a special case of the system \eqref{Al}, precisely, the following
\begin{equation}\label{Al1}
\left\lbrace
\begin{array}{llr}
u_{t} + \alpha u_{x} + \beta u_{xxx} - u_{xxxxx} = f(t,x) & \mbox{in} \ [0,T]\times\mathbb{R}^{+}, \\ 
u(t,0)=u_{x}(t,0)= 0& \mbox{on} \ [0,T], \\
 u(0,x) =0 & \mbox{in} \ \mathbb{R}^{+}.
\end{array}\right. 
\end{equation} 
For the solutions of this system the next lemma holds.

\begin{lemma}Suppose that $ f \in L^{2}(0,T; L^{2}(\mathbb{R}^{+}))$ and $u:= S(0, 0, 0, f)$ is solution of \eqref{Al1}, then
\begin{equation}\label{u e f}
\int_{\mathbb{R}^{+}}^{}|u(t,x)|^{2}dx \leq 2\int_{0}^{t}\!\!\int_{\mathbb{R}^{+}}^{}f(\tau,x)u(\tau,x)dxdt, \quad \forall  t\in [0,T].
\end{equation} 
\end{lemma}
\begin{proof}
Consider $f\in C_{0}^{\infty}(Q^{+}_{T})$ and $u= S(0,0, 0, f_{1})$ a smooth solution of \eqref{Al1}. Multiplying \eqref{Al1} by $2u$, integrating by parts on $[0,R]$, for $R > 0$, yields that
\begin{equation*}
\begin{split}
\frac{d}{dt}\int_{0}^{R}|u(t,x)|^{2}dx =& \ 2\int_{0}^{R}f(t,x)u(t,x)dx - \alpha(|u(t,R)|^{2} - |u(t,0)|^{2}) \\
&+  \beta(|u_{x}(t,R)|^{2} - |u_{x}(t,0)|^{2}) + (|u_{xx}(t,R)|^{2} - |u_{xx}(t,0)|^{2}) \\
&-2\beta (u_{xx}(t,R)u(t,R) - u_{xx}(t,0)u(t,0)) \\
&+ 2(u_{xxxx}(t,R)u(t,R) - u_{xxxx}(t,0)u(t,0))\\
&-2(u_{xxx}(t,R)u_{x}(t,R) - u_{xxx}(t,0)u_{x}(t,0)).
\end{split}
\end{equation*}

So, taking $R \rightarrow \infty$, integrating in $[0,t]$ and using the boundary condition of \eqref{Al1}, we get
$$ \int_{\mathbb{R}^{+}}|u(t,x)|^{2}dx \leq 2\int_{0}^{t}\int_{\mathbb{R}^{+}}f(\tau,x)u(\tau,x)dxd\tau, $$
showing \eqref{u e f} for smooth solutions. The result for the general case follows by density.
\end{proof}


Consider the space $$\widetilde{W}^{1,p}(0,T) = \{\varphi \in W^{1,p}(0,T); \varphi(0)=0\}, \ p \in [2, \infty]$$
and define the following linear operator $Q$
$$Q(u)(t):=q(t),$$
where $q(t)$ is defined by \eqref{q}. Here, we consider the following norm associated to $\widetilde{W}^{1,p}(0,T)$
$$\|Q(u)\|_{\widetilde{W}^{1,p}(0,T)} = \|q\|_{\widetilde{W}^{1,p}(0,T)}= \|q'\|_{L^{p}(0,T)}.$$
With this in hands, we have the following result.
\begin{lemma}\label{lss}Consider $\omega\in J$, defined by \eqref{jota}, and  $\varphi \in \widetilde{W}^{1,p}(0,T)$, for some $p \in [2, \infty],g \in C(0,T; L^{2}(\mathbb{R}^{+}))$. If the following assumption holds
\begin{equation}\label{g0}
\biggl|\int_{\mathbb{R}^{+}}^{}g(t,x)\omega(x)dx\biggr| \geq g_{0} > 0, \ \forall \ t \in [0,T],
\end{equation}
then there exist a unique function $f_{0}= \Gamma(\varphi) \in L^{p}(0,T)$, such that for $f (t, x) := f_0(t)g(t, x)$ the function $u := S(0,0, 0,f)$ solution of  \eqref{Al1} satisfies \eqref{cond. de con.2}. Additionally, the linear operator
\begin{equation}
\begin{array}{c c c c }
\Gamma : & \widetilde{W}^{1,p}(0,T)&  \longrightarrow & L^{p}(0,T) \\
 & \varphi & \longmapsto & \Gamma (\varphi )=f_0
\end{array}
\end{equation}
  is bounded. 
\end{lemma}
\begin{proof}
Consider the function 
$$G: L^{p}(0,T)  \longrightarrow L^{2}(0,T; L^{2}(\mathbb{R}^{+})) $$
defined by $$ f_{0} \longmapsto G(f_{0})= f_{0}g.$$ By the definition, $G$ is linear. Moreover, we have 
\begin{equation*}
\begin{split}
\|G(f_{0})\|^{2}_{L^{2}(0,T;L^{2}(\mathbb{R}^{+}))} \leq&\ \|g\|^{2}_{C(0,T;L^{2}(\mathbb{R}^{+}))}\|f_{0}\|^{2}_{L^{2}(0,T)}\\
\leq &\  T^{\frac{p-2}{p}}\|g\|^{2}_{C(0,T;L^{2}(\mathbb{R}^{+}))}\|f_{0}\|^{2}_{L^{p}(0,T)}.
\end{split}
\end{equation*}
Thus, 
\begin{equation}\label{1.36}
\|G(f_{0})\|_{L^{2}(0,T;L^{2}(\mathbb{R}^{+}))} \leq T^{\frac{p-2}{2p}}\|g\|_{C(0,T;L^{2}(\mathbb{R}^{+}))}\|f_{0}\|_{L^{p}(0,T)}.
\end{equation}
Consider the application
$$\Lambda= Q\circ S \circ G : L^{p}(0,T) \longrightarrow \widetilde{W}^{1,p}(0,T)$$
which one will be defined by $$
f_{0} \longmapsto \Lambda(f_{0})= \displaystyle \int_{\mathbb{R}^{+}}^{}u(t,x)\omega(x)dx,
$$
where  $u:= S(0,0, 0,f)$. Therefore, since $Q$, $S$ and $G$ are linear and bounded, we have that  $\Lambda$ is linear and bounded and have the following property
$$(\Lambda f_{0})(0)= \int_{\mathbb{R}^{+}}^{}u_{0}(x)\omega(x)dx=0$$
that is, $\Lambda$ is well-defined.

Introduce the operator $$\Lambda= A : L^{p}(0,T) \longrightarrow L^{p}(0,T) $$ by $$
f_{0} \longmapsto A(f_{0}) \in L^{p}(0,T),$$
where
$$
(Af_{0})(t)= \frac{\varphi'(t)}{g_{1}(t)} - \frac{1}{g_{1}(t)}\int_{\mathbb{R}^{+}}^{}u(t,x)(\alpha\omega' + \beta\omega''' - \omega''''')dx.
$$
Here, $u= S(0,0, 0,f)$ and $$g_{1}(t)= \int_{\mathbb{R}^{+}}g(t,x)\omega(x)dx,$$ for all  $ t \in [0,T]$. Observe that, using \eqref{q'} $\Lambda(f_{0})= \varphi$ if and only if $f_{0}= A(f_{0})$. 

Now we show that the operator $A$ is a contraction on $L^p(0, T )$, if we choose an
appropriate norm in this space. To show it let us split our prove in two cases. 

\vspace{0.2cm}
 
\noindent\textbf{Case one:} $2 \leq p < \infty $.

\vspace{0.2cm}

Let $f_{01}, f_{02} \in L^{p}(0,T)$, $u_{1}= (S\circ G)f_{01}$ and $u_{2}= (S\circ G)f_{02}$, so thanks to \eqref{u e f} we get 
\begin{equation}\label{123}
\|u_{1}(t,\cdot)- u_{2}(t,\cdot)\|_{L^{2}(\mathbb{R}^{+})} \leq 2\|g\|_{C(0,T; L^{2}(\mathbb{R}^{+}))}\|f_{01} - f_{02}\|_{L^{1}(0,t)}, \ \forall t \in [0,T].
\end{equation}
Consider $\gamma > 0$  and $t \in [0,T]$, using Hölder inequality, we have
\begin{equation*}
\begin{split}
 \biggl|e^{-\gamma t}\bigl((Af_{01})(t)- (Af_{02})(t)\bigr)\biggr|
\leq &\ \frac{e^{-\gamma t}}{|g_{1}(t)|}\int_{\mathbb{R}^{+}}^{}|(u_{1}(t,x) - u_{2}(t,x))(\alpha\omega' + \beta\omega''' - \omega''''')|dx
\\ 
\leq & \frac{e^{-\gamma t}}{g_{0}}\|\alpha \omega' + \beta\omega''' - \omega'''''\|_{L^{2}(\mathbb{R}^{+})}\|u_{1}(t,\cdot) - u_{2}(t,\cdot)\|_{L^{2}(\mathbb{R}^{+})}\\
\leq &  \frac{1}{g_{0}}\|\omega\|_{H^{5}(\mathbb{R}^{+})}e^{-\gamma t}\|u_{1}(t,\cdot)- u_{2}(t,\cdot)\|_{L^{2}(\mathbb{R}^{+})}.
\end{split}
\end{equation*}
Therefore, now, using \eqref{123}, yields that
\begin{equation*}
\begin{split}
\|e^{-\gamma t}(Af_{01}- Af_{02})\|_{L^{p}(0,T)}  \leq&\ \frac{2\|\omega\|_{H^{5}(\mathbb{R}^{+})}\|g\|_{C(0,T; L^{2}(\mathbb{R}^{+}))}}{g_{0}}\biggl(\int_{0}^{T}e^{-\gamma pt}\biggl(\int_{0}^{t}|f_{01}(\tau) - f_{02}(\tau)|d\tau \biggr)^{p}dt\biggr)^{\frac{1}{p}}\\
\leq& \ C\biggl(\int_{0}^{T}e^{-\gamma pt}\biggl(\int_{0}^{T}|f_{01}(\tau) - f_{02}(\tau)|d\tau \biggr)^{p}dt\biggr)^{\frac{1}{p}}.
\end{split}
\end{equation*}

Finally, using the last inequality for $p \in [2, \infty)$, such that $\frac{1}{p}+\frac{1}{p'}=1$, we have
\begin{equation}\label{Lp}
\begin{split}
\|e^{-\gamma t}(Af_{01}- Af_{02})\|_{L^{p}(0,T)}
\leq & \  c_0\biggl(\int_{0}^{T}e^{-\gamma pt}\biggl(\int_{0}^{T}|f_{01}(\tau) - f_{02}(\tau)|d\tau \biggr)^{p}dt\biggr)^{\frac{1}{p}}\\
& \leq c_{0}\left\|e^{-\gamma \tau}\left(f_{01}-f_{02}\right)\right\|_{L^{p}(0, T)}\left[\int_{0}^{T} e^{-p \gamma t}\left(\int_{0}^{t} e^{p^{\prime} \gamma \tau} \mathrm{d} \tau\right)^{p / p^{\prime}} \mathrm{d} t\right]^{1 / p} \\
& \leq \frac{c_{0} T^{1 / p}}{\left(p^{\prime} \gamma\right)^{1 / p^{\prime}}}\left\|e^{-\gamma t}\left(f_{01}-f_{02}\right)\right\|_{L^{p}(0, T)},
\end{split}
\end{equation}
where $ c_0 = c_0(\|\omega\|_{H^{5}(\mathbb{R}^{+})}, g_{0}, \|g\|_{C(0,T; L^{2}(\mathbb{R}^{+}))})$  is defined by 
\begin{equation}\label{c_0}
c_0:=\frac{2}{g_{0}}\|g\|_{C\left([0, T] ; L_{2}\left(\mathbb{R}^+\right)\right)}\left(|\alpha|\left\|\omega^{\prime}\right\|_{L_{2}\left(\mathbb{R}^+\right)}+|\beta|\left\|\omega^{\prime \prime \prime}\right\|_{L_{2}\left(\mathbb{R}^+\right)}+\left\|\omega^{\prime \prime \prime \prime \prime}\right\|_{L_{2}\left(\mathbb{R}^+\right)}\right).
\end{equation}

\vspace{0.2cm}
 
\noindent\textbf{Case two:} $p= \infty$.

\vspace{0.2cm}

In this case, we have 
\begin{equation}\label{Linfty}
\begin{split}
\left\|e^{-\gamma t}\left(A f_{01}-A f_{02}\right)\right\|_{L^{\infty}(0, T)} \leq & c_0 \sup _{t \in[0, T]} e^{-\gamma t}\left\|u_{1}(t, \cdot)-u_{2}(t, \cdot)\right\|_{L^{2}(\mathbb{R}^+)} \\
 \leq & \ c_0 \sup _{t \in[0, T]} e^{-\gamma t}\left\|f_{01}-f_{02}\right\|_{L^{1}(0, t)} \\ \leq &\ \frac{c_0}{\gamma}\left\|e^{-\gamma t}\left(f_{01}-f_{02}\right)\right\|_{L^{\infty}(0, T)},
\end{split}
\end{equation}
where $c_0 = c_0(T,p,\|\omega\|_{H^{5}(\mathbb{R}^{+})}, g_{0}, \|g\|_{C(0,T; L^{2}(\mathbb{R}^{+}))})$ is defined by \eqref{c_0}. 

Therefore,  in both cases for sufficiently large $\gamma$ the operator $A$ is a contraction and, therefore, for any $\varphi \in \widetilde{W}^{1,p}(0,T)$, there exists a unique $f_{0} \in L^{p}(0,T)$ such that $f_{0}= A(f_{0})$, or equivalently, $\varphi= \Lambda(f_{0})$. Thus, follows that  $\Lambda$ is invertible. Due to the Banach theorem its inverse $$\Gamma :L^{p}(0,T) \longmapsto \widetilde{W}^{1,p}(0,T)$$ is bounded. Particularly, 
\begin{equation}\label{1.49}
\|\Gamma \varphi\|_{L^{p}(0,T)} \leq C(T)\|\varphi'\|_{L^{p}(0,T)}.
\end{equation}
\end{proof}

For prove our second main result of this work we need  one refinement of Lemma \ref{lss}.

\begin{lemma}\label{lss1}
Under the hypothesis of Lemma \ref{lss}, if $c_{0} T \leq p^{1 / p} / 2$, $c_{0}$ given by \eqref{c_0}, and $p^{1 / p}=1$ for $p=+\infty$, we have the following estimate
\begin{equation}\label{Gamma}
\|\Gamma \varphi\|_{L_{p}(0, T)} \leq \frac{2}{g_{0}}\left\|\varphi^{\prime}\right\|_{L_{p}(0, T)},
\end{equation}
for the operator $ \Gamma : \widetilde{W}^{1,p}(0,T)\longmapsto L^{p}(0,T)$.
\end{lemma}
\begin{proof}
Since $f_{0}=A f_{0}=\Gamma \varphi$, taking $\gamma=0$, similar as we did in \eqref{Lp}, we get that
$$
\left\|f_{0}-\frac{\varphi^{\prime}}{g_{1}}\right\|_{L^{p}(0, T)} \leq c_{0}\left[\int_{0}^{T}\left(\int_{0}^{t}\left|f_{0}(\tau)\right| \mathrm{d} \tau\right)^{p} \mathrm{~d} t\right]^{1 / p} \leq \frac{c_{0} T}{p^{1 / p}}\left\|f_{0}\right\|_{L^{p}(0, T)},
$$
and in a way analogous to the one made in \eqref{Linfty} we also have
$$
\left\|f_{0}-\frac{\varphi^{\prime}}{g_{1}}\right\|_{L^{\infty}(0, T)} \leq c_{0} \int_{0}^{T}\left|f_{0}(\tau)\right| \mathrm{d} \tau \leq c_{0} T\left\|f_{0}\right\|_{L^{\infty}(0, T)}.
$$
Thus, for $p \in[2,+\infty]$, we get
$$
\| \Gamma \varphi \|_{L^{p}(0, T)} \leq \frac{1}{g_{0}}\ \| \varphi^{\prime} \|_{L^{p}(0, T)}+\frac{c_{0} T}{p^{1 / p}}\| \Gamma \varphi \|_{L^{p}(0, T)},
$$
and the estimate \eqref{Gamma} holds true.
\end{proof}

\section{Control results}\label{Sec3}

In this section the \textit{overdetermination control problem} is studied. Precisely we will give answers for some question left in the beginning of this work. Here, let us consider the full system 
\begin{equation}\label{2aa}
\left\lbrace
\begin{array}{llr}
u_{t} + \alpha u_{x} + \beta u_{xxx} - u_{xxxxx} + uu_{x}= f(t,x)& \mbox{in} \ [0,T]\times\mathbb{R}^{+}, \\ 
u(t,0)=\mu(t), u_{x}(t,0)= \nu(t) & \mbox{on} \ [0,T], \\
 u(0,x) = u_{0}(x) & \mbox{in} \ \mathbb{R}^{+}.
\end{array}\right. 
\end{equation} 
First, we prove that when we have the linear system associated to \eqref{2aa} the control problem with a integral overdetermination condition holds. After that, we are able to extend this result, by using the regularity in Bourgain spaces, to the nonlinear one. Finally, we give, under some hypothesis, a minimal time such that the solution of \eqref{2aa} satisfies \eqref{cond. de con.2}.

\subsection{Linear case} In this section let us present the following result. 
\begin{theorem}\label{caso linear} 
Let $T > 0$, $p \in [2, \infty]$, $u_{0} \in L^{2}(\mathbb{R}^{+})$, $ \mu \in (H^{\frac{2}{5}} \cap L^{p})(0,T)$ and $\nu \in (H^{\frac{1}{5}} \cap L^{p})(0,T).$ Consider $ g \in C(0,T; L^{2}(\mathbb{R}^{+}))$, $\omega \in  \mathcal{J}$, defined by \eqref{jota}, and $ \varphi \in W^{1,p}(0,T)$ such that
\begin{equation}\label{varphi+l}
\varphi(0)= \int_{\mathbb{R}^{+}}u_{0}(x)\omega(x)dx.
\end{equation}
Additionally,  if 
\begin{equation}
\biggl| \int_{\mathbb{R}^{+}}^{}g(t,x)\omega(x)dx\biggr| \geq g_{0} > 0, \ \forall t \in [0,T],
\end{equation}
then there exists a unique $f_{0} \in L^{p}(0,T)$ such that for $f (t, x) := f_0(t)g(t, x) +f_{2x} (t, x) $, with $ f_{2} \in L^{p}(0,T; L^{1}(\mathbb{R}^{+}))$ and $f_{2x} \in X^{0, -b_{0}}(Q^{+}_{T})$, the solution $u:= S(u_{0}, \mu, \nu, f_{0}g + f_{2x})$ of
\begin{equation}\label{2aaa}
\left\lbrace
\begin{array}{llr}
u_{t} + \alpha u_{x} + \beta u_{xxx} - u_{xxxxx} = f(t,x)& \mbox{in} \ [0,T]\times\mathbb{R}^{+}, \\ 
u(t,0)=\mu(t), u_{x}(t,0)= \nu(t) & \mbox{on} \ [0,T], \\
 u(0,x) = u_{0}(x) & \mbox{in} \ \mathbb{R}^{+},
\end{array}\right. 
\end{equation} 
satisfies  \eqref{cond. de con.2}.
\end{theorem}
\begin{proof}
Pick $v_{1}= S(u_{0}, \mu, \nu, -f_{2x})$ solution of
\begin{equation*}
\left\lbrace
\begin{array}{llr}
v_{1t} + \alpha v_{1x} + \beta v_{1xxx} - v_{1xxxxx}= -f_{2x} & \mbox{in} \ Q^{+}_{T}, \\ 
v_{1}(t,0)=\mu(t),  v_{1x}(t,0)= \nu(t) & \mbox{on} \ [0,T], \\
 v_{1}(0,x) = u_{0}(x) & \mbox{in} \ \mathbb{R}^{+}.
\end{array}\right. 
\end{equation*}
Define the following function 
$$
\varphi_{1}= \varphi - Q(v_{1}): [0,T] \longrightarrow \mathbb{R}$$
by $$ \varphi_{1}(t)= \varphi(t) - \int_{\mathbb{R}^{+}}^{}v_{1}(t,x)\omega(x)dx.$$
Since $\varphi \in W^{1,p}(0,T),$ using Lemma \ref{1.1}  together with \eqref{varphi+l}, follows that $\varphi_{1} \in \widetilde{W}^{1,p}(0,T)$. Therefore, Lemma \ref{lss}, ensures that there exists a unique $\Gamma\varphi_{1}= f_{0} \in L^{p}(0,T)$ such that the solution $v_{2}:= S(0,0, 0, f_{0}g)$ of
\begin{equation*}
\left\lbrace
\begin{array}{llr}
v_{2t} + \alpha v_{2x} + \beta v_{2xxx} - v_{2xxxxx}= f_{0}g & \mbox{em} \ Q^{+}_{T}, \\ 
v_{2}(t,0)= 0, v_{2x}(t,0)= 0 & \mbox{em} \ [0,T], \\
v_{2}(0,x) = 0 & \mbox{em} \ \mathbb{R}^{+},
\end{array}\right. 
\end{equation*}
satisfies the following integral condition
$$
\int_{\mathbb{R}^{+}}^{}v_{2}(t,x)\omega(x)dx= \varphi_{1}(t), \  t \in [0,T].
$$

Thus, taking  $u= v_{1} + v_{2}:= S(u_{0},\mu, \nu, f_{0}g - f_{2x})$, we have $u$ solution of \eqref{2aaa} satisfying
\begin{equation*}
\begin{split}
\int_{\mathbb{R}^{+}}^{}u(t,x)\omega(x)dx=& \int_{\mathbb{R}^{+}}^{}v_{1}(t,x)\omega(x)dx + \int_{\mathbb{R}^{+}}^{}v_{2}(t,x)\omega(x)dx\\
=& \int_{\mathbb{R}^{+}}^{}v_{1}(t,x)\omega(x)dx + \varphi_{1}(t)\\
= & \int_{\mathbb{R}^{+}}^{}v_{1}(t,x)\omega(x)dx + \varphi(t) - \int_{\mathbb{R}^{+}}^{}v_{1}(t,x)\omega(x)dx\\
= & \varphi(t), 
\end{split}
\end{equation*}
for all $t\in[0,T]$, that is, \eqref{cond. de con.2} holds, showing the result.
\end{proof}

\subsection{Nonlinear case} We are in position to prove the first main result of this manuscript, that is, Theorem \ref{main1}.  Here,  is essential the estimates in Bourgain space proved by \cite{MC} and presented in the Section \ref{Sec1}.

\begin{proof} [Proof of Theorem \ref{main1}]
Let $u, v \in Z(Q^{+}_{T})$. The following estimate holds, using H\"older inequality, 
\begin{equation*}
\|u(t,\cdot)v(t,\cdot)\|_{L^{1}(\mathbb{R}^{+})} \leq \|u(t,\cdot)\|_{L^{2}(\mathbb{R}^{+})}\|v(t,\cdot)\|_{L^{2}(\mathbb{R}^{+})}, \ \forall t \in [0,T].
\end{equation*} 
So, we get
\begin{equation*}
\|uv\|_{C(0,T;L^{1}(\mathbb{R}^{+}))} \leq \|u\|_{C(0,T;L^{2}(\mathbb{R}^{+}))}\|v\|_{C(0,T;L^{2}(\mathbb{R}^{+}))}.
\end{equation*}
Since we have the following embedding $ \ C(0,T;L^{1}(\mathbb{R}^{+})) \hookrightarrow L^{p}(0,T;L^{1}(\mathbb{R}^{+}))$ for each $p \in [2, \infty]$, we can find 
\begin{equation*}
\|uv\|_{L^{p}(0,T;L^{1}(\mathbb{R}^{+}))} \leq T^{\frac{1}{p}}\|u\|_{C(0,T;L^{2}(\mathbb{R}^{+}))}\|v\|_{C(0,T;L^{2}(\mathbb{R}^{+}))},
\end{equation*}
or equivalently, 
\begin{equation}\label{norma de uv}
\|uv\|_{L^{p}(0,T;L^{1}(\mathbb{R}^{+}))} \leq T^{\frac{1}{p}}\|u\|_{Z(Q^{+}_{T})}\|v\|_{Z_{1}(Q^{+}_{T})},
\end{equation}
for any $ u,v \in Z(Q^{+}_{T})$.

Now, pick $f= f_{1} - f_{2x}$ in the  following system 
\begin{equation}\label{2aaaa}
\left\lbrace
\begin{array}{llr}
u_{t} + \alpha u_{x} + \beta u_{xxx} - u_{xxxxx} = f(t,x)& \mbox{in} \ [0,T]\times\mathbb{R}^{+}, \\ 
u(t,0)=\mu(t), u_{x}(t,0)= \nu(t) & \mbox{on} \ [0,T], \\
 u(0,x) = u_{0}(x) & \mbox{in} \ \mathbb{R}^{+}.
\end{array}\right. 
\end{equation} 
Consider so $f_{2}= \frac{v^{2}}{2}$, where $v \in Z(Q^{+}_{T})$ and $f_{1} \in L^{2}(0,T;L^{2}(\mathbb{R}^{+}))$. The estimate \eqref{dxv2} yields that $f_{2x}= vv_{x} \in X^{0,-b_{0}}(Q^{+}_{T})$, for some $b_{0} \in (0,\frac{1}{2})$. Moreover, thanks to  \eqref{norma de uv} we have that $f_{2} \in L^{p}(0,T;L^{1}(\mathbb{R}^{+})).$

On the space $Z(Q^{+}_{T})$ let us define the functional 
$\Theta: Z(Q^{+}_{T}) \longrightarrow  Z(Q^{+}_{T}) $
by
\begin{equation}\label{Theta}
u:=\Theta v = S\biggl(u_{0}, \mu, \nu, \Gamma\bigl(\varphi - Q(S(u_{0},\mu, \nu , - vv_{x}))\bigr)g - vv_{x}\biggr).
\end{equation}
Note that using Lemma \ref{lss} and Theorem \ref{caso linear}m the operator $\Theta$ is well-defined. 

Considering  $p=2$, thanks to \eqref{norma de u1}, the embedding $L^{2}(0,T;L^{2}(\mathbb{R}^{+}))\hookrightarrow X^{0, -b_{0}}(Q_{T}^{+})$, \eqref{1.36}, \eqref{1.49} and \eqref{norma de q'}, we get

\begin{equation*}
\begin{split}
 \|\Theta v\|_{Z(Q^{+}_{T})} \leq&
C\biggl(\|u_{0}\|_{L^{2}(\mathbb{R}^{+})} + \|\mu\|_{H^{\frac{2}{5}}(0,T)} + \|\nu\|_{H^{\frac{1}{5}}(0,T)}   \\
&+
\|\Gamma\bigl(\varphi - Q(S(u_{0},\mu, \nu , - vv_{x}))\bigr)g  - vv_{x}\|_{X^{0,-b_{0}}(Q^{+}_{T})} \biggr)\\
\leq &
C\biggl(\|u_{0}\|_{L^{2}(\mathbb{R}^{+})} + \|\mu\|_{H^{\frac{2}{5}}(0,T)} +  \|\nu\|_{H^{\frac{1}{5}}(0,T)} +  \|vv_{x}\|_{X^{0,-b_{0}}(Q^{+}_{T})} \\
& + 
\|\Gamma\bigl(\varphi - Q(S(u_{0},\mu, \nu , - vv_{x}))\bigr)g\|_{L^{2}(0,T; L^{2}(\mathbb{R}^{+}))}\biggr)\\
\leq&
C\biggl(\|u_{0}\|_{L^{2}(\mathbb{R}^{+})} + \|\mu\|_{H^{\frac{2}{5}}(0,T)} +  \|\nu\|_{H^{\frac{1}{5}}(0,T)} + \|vv_{x}\|_{X^{0,-b_{0}}(Q^{+}_{T})} 
 \\ 
&+
 \|g\|_{C(0,T; L^{2}(\mathbb{R}^{+}))}\|\bigl(\varphi - Q(S(u_{0},\mu, \nu , - vv_{x}))\bigr)\|_{\widetilde{W}^{1,2}(0,T)} \biggr)\\
\leq&
C(\|g\|_{C(0,T; L^{2}(\mathbb{R}^{+}))}, T)\biggl(\|u_{0}\|_{L^{2}(\mathbb{R}^{+})} + \|\mu\|_{H^{\frac{2}{5}}(0,T)} +  \|\nu\|_{H^{\frac{1}{5}}(0,T)} 
 \\
&+
 \bigl\|vv_{x}\bigr\|_{X^{0, -b_{0}}(Q_{T}^{+})} + \|\varphi'\|_{L^{2}(0,T)} + \|q'\|_{L^{2}(0,T)} \biggr)\\
\leq&
2C(|\alpha|,|\beta|,\|\omega\|_{H^{5}(\mathbb{R}^{+})}, \|g\|_{C(0,T; L^{2}(\mathbb{R}^{+}))}, T)\biggl( c_{1} + \bigl\|vv_{x}\bigr\|_{X^{0, -b_{0}}(Q_{T}^{+})} + \bigl\|v\bigr\|^2_{L^{2}(0,T;L^{1}(\mathbb{R}^{+}))}\biggr),
\end{split}
\end{equation*}
or equivalently,
\begin{equation*}
\|\Theta v\|_{Z(Q^{+}_{T})} \leq 2C(|\alpha|,|\beta|,\|\omega\|_{H^{5}(\mathbb{R}^{+})}, \|g\|_{C(0,T; L^{2}(\mathbb{R}^{+}))}, T)\biggl( c_{1} + \bigl\|vv_{x}\bigr\|_{X^{0, -b_{0}}(Q_{T}^{+})} + \bigl\|v^{2}\bigr\|_{L^{2}(0,T;L^{1}(\mathbb{R}^{+}))}\biggr).
\end{equation*}
Now, using the estimates \eqref{norma de uv} and \eqref{dxv2}, we have that
\begin{equation}\label{Q}
\|\Theta v\|_{Z(Q^{+}_{T})} \leq C\biggl( c_{1} + (T^{\frac{1}{2}} + 1)\|v\|^{2}_{Z(Q_{T}^{+})} \biggr).
\end{equation}
Here, $c_{1}>0$ is a constant depending such that  $$c_1:=\|u_{0}\|_{L^{2}(\mathbb{R}^{+})} + \|\mu\|_{H^{\frac{2}{5}}(0,T)} +  \|\nu\|_{H^{\frac{1}{5}}(0,T)} + \|\varphi'\|_{L^{2}(0,T)}$$ and $C > 0$ is a constant depending of  $C:=C(|\alpha|,|\beta|,\|\omega\|_{H^{5}(\mathbb{R}^{+})}, \|g\|_{C(0,T; L^{2}(\mathbb{R}^{+}))},T)$.

Similarly, using the linearity of the operator $S$, $Q$ and $\Gamma$, once again thanks to  \eqref{norma de uv}  and \eqref{dxv2}, we have 
\begin{equation}\label{Q-Q}
\| \Theta v_{1} - \Theta v_{2}\|_{Z(Q^{+}_{T})} \leq C(T^{\frac{1}{2}} + 1)(\|v_{1}\|_{Z(Q^{+}_{T})} + \|v_{2}\|_{Z(Q^{+}_{T})})\|v_{1} - v_{2}\|_{Z(Q^{+}_{T})}
\end{equation}

Finally, for fixed $c_{1} > 0$, take  $T_{0} > 0$ such that $$8C_{T_{0}}^{2}(T_{0}^{\frac{1}{2}} + 1)c_{1} \leq 1$$ then, for any $T \in (0, T_{0}]$, choice $$r \in \biggl[2C_{T}c_{1}, \frac{1}{\bigl(4C_{T}(T^{\frac{1}{2}} + 1)\bigr)}\biggr].$$ By the other hand, for fixed  $T > 0$  pick $$r= \frac{1}{\bigl(4C_{T}(T^{\frac{1}{2}} + 1)\bigr)}$$ and  $$c_{1} \leq \gamma= \frac{1}{\bigl(8C_{T}^{2}(T^{\frac{1}{2}} + 1)\bigr)}.$$ Therefore,
$$C_{T}c_{1} \leq \frac{r}{2}, \ \ C_{T}(T^{\frac{1}{2}} + 1)r \leq \frac{1}{4}.$$
So, $\Theta$ is a contraction on the ball $B(0,r) \subset Z(Q^{+}_{T})$. Theorem \eqref{caso linear} ensures that the unique fixed point $u=\Theta u \in Z(Q^{+}_{T})$  is a desired solution for $f_{0} := \Gamma\bigl(\varphi - Q(S(u_{0},\mu, \nu, -uu_{x}))\bigr) \in L^p(0,T)$. Thus, the result is achieved. 
\end{proof}

\subsection{Minimal time for the integral condition} We are able now to prove that the integral overdetermination condition \eqref{cond. de con.2a} follows true for a minimal time $T_0$. In order to do that, let us prove the second main result of this work, namely, Theorem \ref{main2}
\begin{proof}[Proof of Theorem \ref{main2}]
Without loss of generality, let us assume $T \leq 1.$ It is well-known that the Kawahara equation \eqref{2aa} enjoys the scaling symmetry: If u is a solution to \eqref{2aa}, $u_{\delta}(t,x)$ defined by  $$u_{\delta}(t,x) := \delta^{4}u(\delta^{5}t, \delta x), \quad \delta>0$$
is solution of \eqref{2aa} as well as.  Indeed, let $\delta= T^{\frac{1}{5}} \in (0,1),$ thus
$$u_{0\delta}(x):= \delta^{4}u_{0}(\delta x), \ \mu_{\delta}(t):= \delta^{4}\mu(\delta^{5}t), \ \nu_{\delta}(t):= \delta^{4}\nu(\delta^{5} t)$$
$$g_{\delta}(t,x):= \delta g(\delta^{5}t,\delta x), \ \omega_{\delta}(x):= \omega(\delta x), \ \varphi_{\delta}(t):= \delta^{4} \varphi(\delta^{5} t).$$
Therefore, if the par $(f_{0}, u)$ is solution of \eqref{2aa}, a straightforward calculation gives that
$$\{f_{0\delta}(t) := \delta^{8}f_{0}(\delta^{5}t), \ u_{\delta}(t,x) := \delta^{4}u(\delta^{5}t, \delta x)\}$$
is solution of 
\begin{equation}\label{2delta}
	\left\lbrace
	\begin{array}{llr}
		u_{\delta t} + \alpha \delta^{4}u_{\delta x} + \beta \delta^{2}u_{\delta xxx} - u_{\delta xxxxx} + u_{\delta}u_{\delta x}= f_{0\delta}(t)g_{\delta}(t,x) & \mbox{in} \ [0,1]\times\mathbb{R}^{+}, \\ 
		u_{\delta}(t,0)=\mu_{\delta}(t), \ u_{\delta x}(t,0)= \nu_{\delta}(t) & \mbox{on} \ [0,1], \\
		u_{\delta}(0,x) = u_{0\delta}(x) & \mbox{in} \ \mathbb{R}^{+}.
	\end{array}\right. 
\end{equation} 
Additionally, we have that $(f_{0}, u)$satisfies \eqref{cond. de con.2} if and only if  $(f_{0\delta}(t) , u_{\delta}(t,x)) $ satisfies the following integral condition
\begin{equation}\label{cond. de con.2 delta}
	\int_{\mathbb{R}^{+}}^{}u_{\delta}(t,x)\omega_{\delta}(x)dx= \varphi_{\delta}(t), \ t \in [0,1].
\end{equation}
Now, using the change of variables theorem and the definition of $\delta$, we verify that
$$\|u_{0\delta}\|_{L^{2}(\mathbb{R}^{+})}= \delta^{\frac{1}{2}}\delta^{4}\|u_{0}\|_{L^{2}(\mathbb{R}^{+})}\leq \delta^{\frac{1}{2}}\|u_{0}\|_{L^{2}(\mathbb{R}^{+})}$$
and
$$\|\varphi'_{\delta}\|_{L^{2}(0,1)}= \delta^{\frac{1}{2}}\delta^{11}\|\varphi_{\delta}\|_{L^{2}(0,T)}\leq \delta^{\frac{1}{2}}\|\varphi_{\delta}\|_{L^{2}(0,T)}.$$
Thus, we have that 
$$c_1(\delta):= \|u_{0\delta}\|_{L^{2}(\mathbb{R}^{+})} +  \|\varphi'_{\delta}\|_{L^{2}(0,1)} + \|\mu_{\delta}\|_{H^{\frac{2}{5}}(0,1)} +  \|\nu_{\delta}\|_{H^{\frac{1}{5}}(0,1)} \leq \delta^{\frac{1}{2}}c_1.$$
Moreover,
$$\|g_{\delta}\|_{C([0,1]; L^{2}(\mathbb{R}^{+}))}\leq \delta^{\frac{1}{2}}\|g\|_{C([0,T]; L^{2}(\mathbb{R}^{+}))},$$
$$ \biggl|\int_{\mathbb{R}^{+}}^{}g_{\delta}(t,x)\omega_{\delta}dx\biggr| \geq g_{0}, \ \forall t \in [0,1],$$
$$\|\omega_{\delta}'\|_{L^{2}(\mathbb{R}^{+})} \leq \delta^{\frac{1}{2}} \|\omega'\|_{L^{2}(\mathbb{R}^{+})},$$
$$\|\omega_{\delta}'''\|_{L^{2}(\mathbb{R}^{+})} \leq \delta^{\frac{5}{2}}\|\omega'''\|_{L^{2}(\mathbb{R}^{+})}$$
and $$\|\omega_{\delta}'''''\|_{L^{2}(\mathbb{R}^{+})} \leq \delta^{\frac{9}{2}}\|\omega''''''\|_{L^{2}(\mathbb{R}^{+})}.$$
So, as we want that $c_{0\delta}$ be one corresponding to $c_0$, which was defined by \eqref{c_0}, therefore,  
$$c_{0\delta} \leq \delta^{5}c_{0}.$$

Pick $\delta_{0}= (2c_{0})^{-1/5},$ so for $0 < \delta \leq \delta_{0}$ we can apply Lemma \eqref{lss1} and according to \eqref{Gamma}, the corresponding operator to $\Gamma$, which one will be called of $\gamma_{\delta}$  satisfies 
\begin{equation}\label{Gamma_delta}
\|\Gamma_{\delta}\varphi_{\delta}\|_{L^{p}(0,1)} \leq \frac{2}{g_{0}}\|\varphi'_{\delta}\|_{L^{p}(0,1)}.	
\end{equation}
Therefore, for $\Theta_{\delta}$ defined in the same way as  in \eqref{Theta} iand using the, similarly as in \eqref{Q} and \eqref{Q-Q}, however, now, using \eqref{Gamma_delta} instead of \eqref{1.49}, we have 
$$
	\|\Theta_{\delta} v_{\delta}\|_{Z(Q^{+}_{1})} \leq C\biggl( \delta^{\frac{1}{2}}c_1 + (T^{\frac{1}{2}} + 1)\|v_{\delta}\|^{2}_{Z(Q^{+}_{1})}\biggr)
$$
and 
$$
	\| \Theta_{\delta} v_{1_{\delta}} - \Theta_{\delta} v_{2_{\delta}}\|_{Z(Q^{+}_{1})} \leq C(T^{\frac{1}{2}} + 1)\bigl(\|v_{1_{\delta}}\|_{Z(Q^{+}_{1}))} + \|v_{2_{\delta}}\|_{Z(Q^{+}_{1})})\|v_{1_{\delta}} - v_{2_{\delta}}\|_{Z(Q^{+}_{1})},
$$
where the constant $C$ is uniform with respect to $0 < \delta \leq \delta_{0}.$ Taking $\delta_{0}$ small enough, if necessary, in order to satisfies the following inequality $$\delta_{0}^{\frac{1}{2}}c_1 \leq \frac{1}{8c^{2}(T^{\frac{1}{2}} + 1)},$$ so using the same arguments as done in Theorem \ref{main1}, the operator $\Theta_{\delta}$ becomes, at least, locally, a contraction on a certain ball. Lastly, taking the time $T_{0}$ defined by $T_0:= \delta_{0}^{5}$, and if $T \leq T_0$ we have that \eqref{cond. de con.2} holds true, showing so the result.
\end{proof}

\subsection{An exact controllability result}
The goal of this subsection is to prove the Corollary \ref{teocontrol}, showing that if the overdetermination condition is verified, for given any initial data $u_0$ and final data $u_T$ the mass \eqref{mass} of the system \eqref{2_1} is reached on the time $T$.

\begin{proof}[Proof of Corollary \ref{teocontrol}]
Thanks to the Theorem \ref{main1} with $\mu=v=0$, there exist $f_{0} \in L^{p}(0,T)$ and a unique solution $u$ of  \eqref{2_1} such that
\begin{equation}\label{eqq2}
\varphi(t)=\int_{\R^+} u(t,x)\omega(x)dx, \quad t \in [0,T]. 
\end{equation}
On the other hand, we know that $\omega$ defined a measure in $\R^+$ given by 
\begin{equation*}
\eta(E)=\int_{E} w(x)dx,
\end{equation*}
for all Lebesgue measure set  $E$ of $\R^+$ and 
\begin{equation*}
\int_{\R^+} f d\eta =  \int_{\R^+} f(x) \omega(x)dx,
\end{equation*}
for all measurable function $f$ in $\R^+$. Hence, from \eqref{eqq1} and \eqref{eqq2}, we conclude that 
\begin{equation*}
[u(T)]= \int_{\R^+} u(T)d\eta = \int_{\R^+} u(T,x)\omega(x)dx=\int_{\R^+} u_T(x)\omega(x)dx=\int_{\R^+} u_Td \eta = [u_T],
\end{equation*}
and the corollary is achieved.
\end{proof}

\section{Further comments} \label{Sec4}
This work deals with the internal controllability problem with an integral overdetermination condition on unbounded domains. Precisely,  we consider the higher order KdV type equation, so-called, Kawahara equation on the right half-line  
\begin{equation}\label{2a}
 \left\lbrace \begin{array}{llr} u_{t} + \alpha u_{x} + \beta u_{xxx} - u_{xxxxx} + uu_{x}= f(t,x)& \mbox{in} \ [0,T]\times\mathbb{R}^{+}, \\  u(t,0)=\mu(t), \ u_{x}(t,0)= \nu(t) & \mbox{on} \ [0,T], \\  u(0,x) = u_{0}(x) & \mbox{in} \ \mathbb{R}^{+}, 
 \end{array}\right. 
  \end{equation} 
  where $f:=f_0(t)g(x,t)$, with $f_0$ as a control input. In this case, we prove that given functions $u_0$, $\mu$, $\nu$ and $g$, the following integral overdetermination condition  
  \begin{equation}\label{cond. de con.2a} 
  \int_{\mathbb{R}^{+}}^{}u(t,x)\omega(x)dx= \varphi(t), \ t \in [0,T], 
  \end{equation} 
  holds. Additionally, that condition can be verified for a small time $T_0$. These points answer the previous questions introduced in \cite{CaGo} and extend for others domains the results of \cite{CaSo}.

\subsection{Comments about the main results} Let us give some remarks in order with respect to the generality of this manuscript.
 \begin{itemize} 
 \item Theorems \ref{main1} and \ref{main2} can be obtained for more general nonlinearity $u^2u_x$. In fact, this is possible due the result of Cavalcante and Kawak \cite{MC1} that showed the following:
 \begin{theorem}  The following estimates holds.
 \begin{itemize}  \item[a)]For $-1 / 4 \leq s$, there exists $b=b(s)<1 / 2$ such that for all $\alpha>1 / 2$, we have $$ \left\|\partial_{x}(u v w)\right\|_{X^{s,-b}} \lesssim\|u\|_{X^{s, b} \cap D^{\alpha}}\|v\|_{X^{s, b} \cap D^{\alpha}}\|w\|_{X^{s, b} \cap D^{\alpha}}. $$ \item[b)]  For $-1 / 4 \leq s \leq 0$, there exists $b=b(s)<1 / 2$ such that for all $\alpha>1 / 2$, we have $$ \left\|\partial_{x}(u v w)\right\|_{Y^{s,-b}} \lesssim\|u\|_{X^{s, b} \cap D^{\alpha}}\|v\|_{X^{s, b} \cap D^{\alpha}}\|w\|_{X^{s, b} \cap D^{\alpha}}. $$ \end{itemize} 
  \end{theorem} 
  Thus, Theorems \ref{main1} and \ref{main2} remain valid for $u^2u_x$, however, for sake of simplicity, we consider only the nonlinearity as $uu_x$. 
   \vspace{0.1cm}
   \item Due to the boundary traces defined in \cite[Theorems 1.1 and 1.2]{MC}, the regularities of the functions involved in this manuscript are sharps. 
    \vspace{0.1cm}
   \item The results presented in this manuscript are still valid when we consider the following domains: the real line ($\mathbb{R}$) and the left half-line ($\mathbb{R}^-$). Precisely, let us consider the following systems 
\begin{equation}\label{1}
\left\lbrace
\begin{array}{llr}
u_{t} + \alpha u_{x} + \beta u_{xxx} - u_{xxxxx} + uu_{x}= f_{0}(t)g(t,x) & \mbox{in} \ [0,T]\times\mathbb{R}, \\ 
 u(0,x) = u_{0}(x) & \mbox{on} \ \mathbb{R}.
\end{array}\right. 
\end{equation} 
and
\begin{equation}\label{3}
\left\lbrace
\begin{array}{llr}
u_{t} + \alpha u_{x} + \beta u_{xxx} - u_{xxxxx} + uu_{x}= f_{0}(t)g(t,x) & \mbox{in} \ [0,T]\times\mathbb{R}^{-}, \\ 
u(t,0)=\mu(t), \ u_{x}(t,0)= \nu(t) , \  u_{xx}(t,0)= h(t) & \mbox{on} \ [0,T], \\
u(0,x) = u_{0}(x) & \mbox{in} \ \mathbb{R}^{-}.
\end{array}\right. 
\end{equation} 
For given $T > 0$, $\varphi$, $\omega$ and $\omega^-$, consider the following integral conditions
\begin{equation}\label{cond. de con.1}
\int_{\mathbb{R}}^{}u(t,x)\omega(x)dx= \varphi(t), t \in [0,T]
\end{equation}
and 
\begin{equation}\label{cond. de con.3}
\int_{\mathbb{R}^{-}}^{}u(t,x)\omega^-(x)dx= \varphi(t), t \in [0,T].
\end{equation}
Thus, the next two theorems give us answers for the Problems $\mathcal{A}$ and $\mathcal{B}$, presented in the beginning of the manuscript, for real line and left half-line, respectively.
\begin{theorem}\label{main3} Let $T > 0$ and $p \in [2,\infty]$. Consider $u_{0} \in L^{2}(\mathbb{R})$ and $\varphi \in W^{1,p}(0,T)$. Additionally, let  $g \in C(0,T; L^{2}(\mathbb{R}))$ and  $\omega \in H^{5}(\mathbb{R})$ be a fixed function
 satisfying
$$
\varphi(0)= \int_{\mathbb{R}}u_{0}(x)\omega(x)dx
$$
and
$$
\biggl|\int_{\mathbb{R}}g(t,x)\omega(x)dx\biggr| \geq g_{0} > 0, \ \forall t \in [0,T],
$$
where $g_{0}$ is a constant. Then, for each $T > 0$ fixed, there exists a constant $\gamma > 0$ such that if  $c_{1}= \|u_{0}\|_{L^{2}(\mathbb{R}} + \|\varphi'\|_{L^{2}(0,T)} \leq \gamma,$ we can find a unique control input $f_{0} \in L^{p}(0,T)$ and a unique solution $u$ of  \eqref{1} satisfying  \eqref{cond. de con.1}.
\end{theorem}

\begin{theorem}\label{main4} Let $T > 0$ and $p \in [2,\infty]$. Consider $\mu \in H^{\frac{2}{5}}(0,T) \cap L^{p}(0,T)$, $\nu \in H^{\frac{1}{5}}(0,T) \cap L^{p}(0,T)$, $h\in L^p(0,T)$, $u_{0} \in L^{2}(\mathbb{R}^{-})$ and $\varphi \in W^{1,p}(0,T)$. Additionally, let  $g \in C(0,T; L^{2}(\mathbb{R}^{-}))$ and  $\omega^-$ be a fixed function which belongs to the following set 
\begin{equation}\label{jota1}
\mathcal{J}= \{\omega \in H^{5}(\mathbb{R}^-): \omega(0)= \omega'(0)=0 \}
\end{equation}
 satisfying
$$
\varphi(0)= \int_{\mathbb{R}^{-}}u_{0}(x)\omega^-(x)dx
$$
and
$$
\biggl|\int_{\mathbb{R}^{-}}g(t,x)\omega^-(x)dx\biggr| \geq g_{0} > 0, \ \forall t \in [0,T],
$$
where $g_{0}$ is a constant. Then, for each $T > 0$ fixed, there exists a constant $\gamma > 0$ such that if 
$$c_{1}= \|u_{0}\|_{L^{2}(\mathbb{R}^{-})} + \|\mu\|_{H^{\frac{2}{5}}(0,T)} +  \|\nu\|_{H^{\frac{1}{5}}(0,T)} + \|h\|_{L^2(0,T)}+\|\varphi'\|_{L^{2}(0,T)} \leq \gamma,$$ we can find a unique control input $f_{0} \in L^{p}(0,T)$ and a unique solution $u$ of  \eqref{3} satisfying  \eqref{cond. de con.3}.
\end{theorem}
   \item The difference between the numbers of boundary conditions in \eqref{2a} and \eqref{3} is motivated by integral identities on smooth solutions to the linear Kawahara equation
$$u_t+\alpha u_{x}+\beta u_{xxx}-u_{xxxxx}=0.$$
   \item Theorem \ref{main2} is also true for the systems \eqref{1} and \eqref{3}. Additionally,  due the results presented in \cite{MC1,MC} the functions involved in Theorems \ref{main3} and \ref{main4} are also sharp and we can introduce a more general nonlinearity like $u^2u_x$ in these systems.
       \vspace{0.1cm}
    \item Corollary \ref{teocontrol} may be extended for the system \eqref{1} taking into account the integral condition \eqref{cond. de con.1}. Also for the system \eqref{3}, with $u(t,0)=u_{x}(t,0)= u_{xx}(t,0)= 0$ and the integral condition \eqref{cond. de con.3}, this corollary is verified.
\end{itemize}

\subsection{General control result} Finally, we would like to comment about a more general control result.  Thanks to the Corollary \ref{teocontrol} it is possible to obtain an exact controllability property related to the mass of the system.  However, we would like to show the following exact controllability result:

\vspace{0.2cm}

\noindent\textbf{Exact control problem:} \textit{Given $u_0, u_T \in L^2(\R^+)$ and  $g \in C(0,T; L^{2}(\mathbb{R}^{+}))$ satisfying \eqref{damping},  can we find a control $f_0 \in L^p(0,T)$ such that the solution $u$ of \eqref{2_1} satisfies $u(T,x)=u_T(x)$?}

\vspace{0.2cm}

A possibility to give an answer for this question is to modify the overdetermination condition \eqref{cond. de con.2}.  For example, if Theorem \ref{main1}  is verified for the following integral condition
   \begin{equation}\label{newcondition}
       \widetilde\varphi(t)=\int_{\R^+}u^2(t,x)w(x)dx,
   \end{equation}
   we are able to get the exact controllability in $L^2(\R^+)$ with internal control  $f_0\in L^2(0,T)$ by using the same argument as in Corollary \ref{teocontrol}.  However, with the approach used in this manuscript, it is not clear that the Lemma \ref{lss} can be replied for the condition \eqref{newcondition}.

Indeed,  if we consider
   \begin{equation*}
       q(t)=\int_{\R^+}u^2(t,x)w(x)dx,
   \end{equation*}
analyzing $q'(t)$ for $u=S(0,0,0,f_0(t)g(t,x))$ (see Lemma \ref{q}) we obtain
   \begin{equation*}
   \begin{split}
       q'(t)=&\int_{\R^+} u^2(t,x)\left[ \alpha w'(x)+\beta w''(x) - 2w'''''(x)\right]dx \\
      & +\int_{\R^+} u^2_x(t,x)\left[5 w'''(x)-3\beta w'(x) - 2w'''''(x)\right]dx \\
       &- 5\int_{\R^+} u^2_{xx}(t,x) w'(x)dx +f_0(t) \int_{\R^+} g(t,x)u(t,x) w(x)dx.
       \end{split}
   \end{equation*}
   Now, introduce the operator $$\widetilde A : L^{p}(0,T) \longrightarrow L^{p}(0,T) $$ defined by $$
f_{0} \longmapsto \widetilde A(f_{0}) \in L^{p}(0,T),$$
where
   \begin{equation*}
   \begin{split}
(\widetilde Af_{0})(t)=& \varphi'(t) -\int_{\R^+} u^2(t,x)\left[ \alpha w'(x)+\beta w''(x) - 2w'''''(x)\right]dx \\
     &  -\int_{\R^+} u^2_x(t,x)\left[5 w'''(x)-3\beta w'(x) - 2w'''''(x)\right]dx        + 5\int_{\R^+} u^2_{xx}(t,x) w'(x)dx.
       \end{split}
   \end{equation*}
If we assume that $\Lambda (f_0)=\widetilde \varphi$, we deduce that 
\begin{equation*}
    (\widetilde Af_{0})(t) = f_0(t) \int_{\R^+} g(t,x)u(t,x) w(x)dx.
\end{equation*}
Note that this expression depends of solution, then we do not be able to obtain the overdetermination control condition for $S(0,0,0,f_0(t)g(t,x))$ by using a fixed point argument for the operator $$\left[\int_{\R^+} g(t,x)u(t,x) w(x)dx\right]^{-1}(\widetilde Af_{0})(t),$$
as in the proof of Lemma \ref{lss}. Therefore, the exact controllability with internal control does not holds.  Hence, the following open question arises: 

\vspace{0.2cm}
\noindent\textbf{Question:} \textit{Is it possible to prove Theorem \ref{main1} for the overdetermination condition \eqref{newcondition}?}
   
\subsection*{Acknowledgments:} 

Capistrano--Filho was supported by CNPq grant 408181/2018-4, CAPES-PRINT grant  88881.311964/2018-01, MATHAMSUD grants  88881.520205/2020-01, 21-MATH-03 and Propesqi (UFPE). Gallego was supported by MATHAMSUD 21-MATH-03 and the 100.000 Strong in the Americas Innovation Fund. de Sousa acknowledges support from CAPES-Brazil and CNPq-Brazil. This work is part of the PhD thesis of de Sousa at Department of Mathematics of the Universidade Federal de Pernambuco which was was supported by CNPq and Capes.

\subsection*{Data availability statement:} The data that support the findings of this study are available from the corresponding authorupon request.

\end{document}